\documentclass[11pt,reqno,a4paper,english]{amsart}

\usepackage{a4wide}
\usepackage[utf8]{inputenc}  
\usepackage{dsfont}
\usepackage{xcolor}
\usepackage{comment}
\usepackage{enumerate}          
\usepackage{amssymb,amsmath,amsfonts,amsthm}            
\usepackage{pdfsync}
\usepackage{url}
\usepackage{hyperref}
\usepackage{mathrsfs}
\usepackage{graphicx,float,color}
\usepackage{epsfig}
\usepackage{cleveref}
\usepackage{microtype}
\usepackage{cancel}

\begin{document}
\theoremstyle{plain}
\newtheorem{theorem}{Theorem}[section]
\newtheorem{definition}[theorem]{Definition}
\newtheorem{proposition}[theorem]{Proposition}
\newtheorem{lemma}[theorem]{Lemma}
\newtheorem{corollary}[theorem]{Corollary}
\newtheorem{conjecture}[theorem]{Conjecture}
\newtheorem{remark}[theorem]{Remark}
\errorcontextlines=0
\newcommand{\SW}{\mathrm{SW}}
\newcommand{\KME}{\mathrm{KME}}
\newcommand{\LOT}{\mathrm{LOT}}
\renewcommand{\P}{\mathcal{P}}
\newcommand{\cst}{\mathrm{cst}}
\newcommand{\Id}{\text{Id}}
\newcommand{\R}{\mathbb{R}}
\newcommand{\sgn}{\mathrm{sgn}}
\newcommand{\M}{\mathbf{H}^m}
\newcommand{\Op}{{\rm Op}}
\newcommand{\J}{\mathcal{J}}
\newcommand{\Z}{\mathbb{Z}}
\newcommand{\N}{\mathbb{N}}
\newcommand{\supp}{\text{supp}}
\newcommand{\ord}{\text{ord}}
\newcommand{\X}{\mathcal{X}}
\newcommand{\e}{\varepsilon}
\newcommand{\bqn}{\begin{equation}}
\newcommand{\eqn}{\end{equation}}
\newcommand{\al}{\alpha}
\newcommand{\tor}{(2\mathbb{T})^{2}}
\newcommand{\er}{{\overrightarrow{e_{r}}}}
\newcommand{\et}{{\overrightarrow{e_{\theta}}}}
\newcommand{\red}{\color{red}}
\newcommand{\black}{\color{black}}
\renewcommand{\i}{\iota}
\newcommand{\todo}[1]{$\clubsuit$ {\tt #1}}
\newcommand{\B}{\mathcal{B}}
\newcommand{\F}{\mathbb{F}}
\newcommand{\Mom}{\mathcal{M}}
\newcommand{\Cay}{{\rm Cay}}
\newcommand{\spt}{\mathrm{spt}}
\newcommand{\Lip}{\mathrm{Lip}}
\newcommand{\Prob}{\mathcal P}
\newcommand{\Wass}{W}
\newcommand{\SlicedWass}{SW}
\newcommand{\proj}{proj}

\newcommand{\Rsp}{\mathbb R}
\newcommand{\Sph}{\mathbb S}
\newcommand{\dd}{\mathrm d}
\newcommand{\LL}{L}
\newcommand{\1}{\mathbbm 1}

\newcommand{\pushforward}[2]{#1_{\#}#2}
\newcommand{\sca}[2]{\left\langle #1,#2\right\rangle}
\newcommand{\abs}[1]{\left|#1\right|}
\newcommand{\nr}[1]{\left\Vert#1\right\Vert}

\setcounter{tocdepth}{1}

\title{Embedding large subsets of Wasserstein spaces into~Banach~spaces}

\author{Cyril Letrouit}\address{Cyril Letrouit. Université Paris-Saclay, CNRS, Inria, Laboratoire de mathématiques d’Orsay, 91405, Orsay, France} \email{cyril.letrouit@universite-paris-saclay.fr}

\author{Quentin Mérigot}\address{Quentin Mérigot. Université Paris-Saclay, CNRS, Inria, Laboratoire de mathématiques d’Orsay, 91405, Orsay, France /
DMA, École normale supérieure, Université PSL, CNRS, 75005 Paris, France} \email{quentin.merigot@universite-paris-saclay.fr}
\date{\today}

\maketitle

\begin{abstract}
Embedding subsets of Wasserstein spaces into linear spaces is a problem of both theoretical and computational interest. 
In this work, we prove that three classical and natural embeddings of Wasserstein spaces into Banach spaces of non-trivial type are bi-Hölder on subsets of Wasserstein spaces given by moment bounds. These embeddings, all considered in multiple applications, are the linearized optimal transport embedding, the sliced-Wasserstein embedding, and a particular kernel mean embedding. 
Our results contrast with known obstructions to embeddin the full Wasserstein space into Banach spaces of non-trivial type. We also show the sharpness of our results regarding the moments required for our conclusions to hold.
\end{abstract}

\section{Introduction}
Optimal transport has found many applications both within and outside mathematics. One of its most useful features is that it allows one to define geometric distances, the so-called Wasserstein distances, which can be used to compare probability measures on a fixed metric space. In this article, we work on the Euclidean space $\Rsp^d$ and endow the space of probability measures $\Prob_p(\Rsp^d)$ with finite $p$-moment with the $p$-Wasserstein distance, denoted $\Wass_p$. The use of the metric space $(\Prob_p(\Rsp^d),\Wass_p)$ in practical applications, especially in data science, comes with two difficulties. First, in dimension $d>1$, Wasserstein spaces are not flat, which forbids a direct use of linear statistical methods such as principal component analysis. Second, despite the tremendous progress in optimal transport solvers in recent years, computing an optimal transport plan is still costly. These two issues could be alleviated if one could construct embeddings $\Phi$ of $\Prob_p(\Rsp^d)$ into simpler spaces such as $\LL^q$ spaces ($1\leq q<+\infty$), or more generally into Banach spaces of non-trivial type (see \cite[Chapter 9]{ledoux1991probability} for types of Banach spaces). One would like these embeddings to preserve as well as possible the geometric structure of the Wasserstein space; ideally they should be bi-Lipschitz or bi-Hölder.

In dimension $d=1$, this problem admits a positive and ideal solution: the Wasserstein space $(\Prob_p(\Rsp), \Wass_p)$ can be embedded \emph{isometrically} into a subset of $\LL^p([0,1])$. To do so, one associates to any measure $\mu\in\Prob(\Rsp)$ its quantile function $Q_{\mu}:(0,1)\to\Rsp$, which is the generalized inverse of the cumulative distribution function $F_\mu$: 
\begin{equation}
    Q_\mu(m) = \inf\{t\in\Rsp \mid F_\mu(t)\geq m\}. 
\end{equation}
These quantile functions satisfy, for all $\mu,\nu$ in $\Prob_p(\Rsp)$, $\nr{Q_\mu - Q_{\nu}}_{\LL^p([0,1])} = \Wass_p(\mu,\nu)$, thus providing the isometric embedding. 

In higher dimension, however, constructing such an isometric embedding becomes impossible because of the (positive) curvature of the Wasserstein spaces.  An important body of work quantifies the hardness of approximating Wasserstein distances using embeddings, see for instance \cite[Section 8.3.1]{peyrecuturi} for a short review. Most work have focused on $W_1$ \cite{naor2007planar,gartland2026lower}, although $W_2$ is often a more suitable choice in many application areas, e.g. in computer vision, graphics, and machine learning (see \cite{courty2018learning} and references therein). One reason for the lack of results for $W_2$ (and for $W_p$, $p>1$) is a general result by Andoni, Naor and Neiman \cite{ann}, who showed the impossibility to embed Wasserstein spaces (for $p>1$) into Banach spaces of non-trivial type, even in a very weak sense. Their statement rules out all coarse embeddings (see \cite[Section 1.4.2]{ann} for a definition), in particular all uniform embeddings. Recall that an embedding $\Phi$ is said to be uniform, or bi-uniformly continuous, if $\Phi$ is injective, and $\Phi$ and its inverse are uniformly continuous. For instance, any bi-Hölder embedding is uniform, hence coarse.

\begin{theorem}[Andoni, Naor, Neiman \cite{ann}]\label{t:ann} Let $p>1$ and $d\geq 3$. Then, there is no coarse embedding of $(\Prob_p(\R^d),\Wass_p)$ into any Banach space of non-trivial type\footnote{Specifying that the target Banach spaces have non-trivial type is necessary, since every metric space $(X,d)$, in particular Wasserstein spaces, admits a tautological isometric embedding, called Kuratowski embedding, into the Banach space of bounded continuous functions on $X$ equipped with the sup-norm, which has trivial type as soon as $X$ is infinite. The Kuratowski embedding works for general metric spaces, but does not usually simplify computations or geometry.}. In particular, $(\Prob_p(\Rsp^d), \Wass_p)$ does not admit any bi-Hölder embedding into any $\LL^q$ space for $q\in(1,+\infty)$.
\end{theorem}

In contrast to the negative result given by Theorem \ref{t:ann}, our goal in this paper is to show that classical and natural embeddings of Wasserstein spaces, considered in applications, are bi-Hölder continuous on subsets of Wasserstein spaces (that we seek as large as possible) into Banach spaces of non-trivial type. In other words, Theorem \ref{t:ann} is an obstruction to embedding \emph{full} Wasserstein spaces, but we show that it is nevertheless possible to embed large subsets of Wasserstein spaces in a controlled way.

Three embeddings are  especially relevant for applications: the \emph{linearized optimal transport embedding} \cite{wangslepcev}, the \emph{sliced-Wasserstein embedding} \cite{rabinpeyredelonbernot} and the \emph{kernel-mean embedding} \cite{gretton1}, respectively abbreviated by LOT, SW and KME (see \cite{merigot2026embedding} for a review of LOT and SW embeddings).
For suitable parameters, these three constructions define maps from probability measures into Banach spaces: 
\[
\begin{aligned}
\Phi^{\LOT}_{p,\rho}&:
\Prob_p(\Rsp^d)\longrightarrow \LL^p(\rho,\Rsp^d),\\
\Phi^{\SW}_p&:
\Prob_p(\Rsp^d)\longrightarrow
\LL^p(\Sph^{d-1}\times[0,1]),\\
\Phi^{\KME}_\beta&:
\Prob_1(\Rsp^d)\longrightarrow \mathcal H_\beta,
\qquad \beta\in(0,1).
\end{aligned}
\]
For the LOT embedding, we fix a reference measure $\rho$ belonging to the set $\Prob_{\mathrm{ac}}(\Rsp^d)$ of absolutely continuous measures on $\R^d$, and we assume $\rho$ is compactly supported. The target space is $\LL^p(\rho,\Rsp^d)$, the space of $\R^d$-valued function which are $p$-integrable with respect to $\rho$. For the KME embedding, $\mathcal H_\beta$ is the reproducing kernel Hilbert space associated with the kernel 
$$
\kappa_\beta(x,y) =|x|^{2\beta}+|y|^{2\beta}-|x-y|^{2\beta},
$$
for $\beta\in (0,1)$. 

When they are injective, these embeddings can be used to define a distance on the Wasserstein space, by pullback,  that can then be used as an alternative to the Wasserstein distance. It is already known that, when restricted to the set $\Prob(B(0,R))\subseteq \Prob_p(\Rsp^d)$, these embeddings are bi-Hölder on their image. As an illustration, we quote the following result from  Bonnotte \cite{bonnotte} for the sliced-Wasserstein embedding.

\begin{theorem}[Bonnotte] \label{th:bonnotte}
Given $R>0$ and $p\geq1$, the restriction of the embedding $\Phi_{p}^{\SW}$ to $\Prob(B(0,R))$ is bi-Hölder on its image. More precisely, there exists a constant $C = \cst(p,d,R)$ such that for all $\mu,\nu\in\Prob(B(0,R))$,
$$ C \Wass_p(\mu,\nu)^{p(d+1)}\leq \nr{\Phi_p^{\SW}(\mu) - \Phi_p^{\SW}(\nu)}_{\LL^p(\Sph^{d-1}\times[0,1])} \leq \Wass_p(\mu,\nu),$$
\end{theorem}
The contrast between the obstruction of Theorem \ref{t:ann} on $\Prob_p(\Rsp^d)$ and these positive results on $\Prob(B(0,R))$ for the three embeddings defined above naturally leads to the question of understanding how these embeddings behave on intermediate classes of measures, in which non-compactly supported measures are allowed but the behaviour of tails is controlled. We will be especially interested in the following sets of probability measures whose moment of order $q\geq p$ is bounded:
$$ \Mom_{q,R} = \{ \mu \in\Prob(\Rsp^d) \mid M_q(\mu) \leq R^q\}, \hbox{ where } M_q(\mu) = \int \abs{x}^q\dd\mu(x). $$
We note that these classes interpolate between $\Prob(B(0,R))$ and $\Prob_p(\Rsp^d)$. More precisely
$$ \Mom_{+\infty,R} := \bigcap_{p>0} \Mom_{p,R} = \Prob(B(0,R)),\hbox{ while } \Mom_{p,+\infty} := \bigcup_{R>0} \Mom_{p,R} = \Prob_p(\Rsp^d). $$ 

Our main result identifies a sharp threshold in the order $q$ of the moments that must be controlled for these embeddings to faithfully reflect the Wasserstein geometry. This common threshold for three rather different embeddings is the main phenomenon highlighted by our results.

\begin{theorem}[Informal]
Let $p>1$, $\beta\in(0,1)$ and $R>0$. The restrictions of the embeddings $\Phi^\LOT_{p,\rho},\Phi^\KME_\beta$ and $\Phi^\SW_p$ to the set $\Mom_{q,R}$, viewed as a subset of $(\Prob_p(\Rsp^d),\Wass_p)$, are bi-Hölder onto their images whenever $q>p$. In contrast, at the critical exponent $q=p$, the embeddings are not even bi-uniformly continuous on $\Mom_{p,R}$ --- for $\Phi^\SW_p$, this is established under the additional assumption that $p\in\mathbb{N}$.
\end{theorem}

In the remainder of this introduction, we describe the three embeddings and state the precise results obtained for each of them.

\subsection{Embedding via linearized optimal transport}
Throughout this section, we fix an absolutely continuous probability measure $\rho \in \Prob_{\mathrm{ac}}(\Rsp^d)$, which we assume  for simplicity to be supported on a compact and convex set. For any $p>1$ and any measure $\mu\in\Prob_p(\Rsp^d)$, the theorem of Gangbo-McCann \cite{gangbomccann} ensures that the optimal transport between $\rho$ and $\mu$ for the cost $c_p(x,y) = \abs{x-y}^p$ is unique and  induced by a map. We denote this map by $T_{\rho\to\mu,p}$ and we consider it as a vector field in $\LL^p(\rho,\Rsp^d)$. The idea behind the \emph{linearized optimal transport} embedding \cite{wangslepcev}, initially proposed in the case $p=2$, is to represent any measure $\mu \in\Prob_p(\Rsp^d)$ by its $p$-optimal transport map from the fixed measure $\rho$. The resulting map 
$$\Phi^\LOT_{p,\rho}: \Prob_p(\Rsp^d) \rightarrow \LL^p(\rho,\R^d), \qquad \mu \mapsto T_{\rho\to\mu,p}$$
is injective. Moreover, a coupling argument even shows that it does not decrease  distances: 
$$ \forall \mu,\nu\in\Prob_p(\Rsp^d),\qquad \Wass_p(\mu,\nu) \leq \nr{\Phi^\LOT_{p,\rho}(\mu) - \Phi^\LOT_{p,\rho}(\nu)}_{\LL^p(\rho,\Rsp^d)}.$$
Proving that the LOT embedding is bi-Hölder over a certain class $\Mom_{q,R}$ of probability measures therefore amounts to establishing a quantitative stability estimate for optimal transport maps on this class. For measures supported over a fixed ball of radius $R$, corresponding to limit class $\Mom_{+\infty,R}$, such estimates are already known. In the quadratic case $p=2$, we refer to Berman \cite{berman}, Delalande-Mérigot \cite{delalandemerigot}, Letrouit-Mérigot \cite{letrouitmerigot} and Mérigot \cite{merigot},  while  for the cost $c_p$ with exponent $p>1$, stability is established by Mischler-Trevisan \cite{mischlertrevisan}. 
In the quadratic case, the stability estimates for $\Mom_{+\infty,R}$ have been extended to the class $\Mom_{q,R}$ in \cite{delalandemerigot} but only for sufficiently large exponents $q$, satisfying both $q>d$ and $q\geq 4$. Moreover, the Hölder exponent obtained in this case depends on the dimension and deteriorates when $q$ tends to $d$. 

We establish here that the restriction of the embedding $\Phi^\LOT_{2,\rho}$ to the class $\Mom_{q,R}$ is bi-Hölder as soon as $q>2$:

\begin{theorem} \label{t:p>2}
Let $\rho \in\Prob_{\mathrm{ac}}(\Rsp^d)$ and assume that its density is bounded from above and below by positive constants on its support, which is compact and convex. Let $q>2$ and $R\geq 0$. Then for some constant $C=\cst(q,R)$,
\begin{equation}\label{e:2q-1}
\forall \mu,\nu \in \Mom_{q,R}, \qquad W_2(\mu,\nu)\leq \|T_{\rho\to\mu,2}-T_{\rho\to\nu,2}\|_{L^2(\rho)}\leq C \Wass_1(\mu,\nu)^{\frac14-\frac{3}{4(2q-1)}}
\end{equation}
Consequently, $\Phi_{2,\rho}^\LOT$ is a bi-Hölder embedding of $\Mom_{q,R}\subseteq \Prob_2(\Rsp^d)$ into the Hilbert space $\LL^2(\rho,\Rsp^d)$. 
Additionally, the exponent in the right-hand side of \eqref{e:2q-1} is optimal.
\end{theorem}

The first sentence below \eqref{e:2q-1} directly follows from \eqref{e:2q-1} using $W_1\leq W_2$.
An important feature of this result is that the exponent does \emph{not} depend on the dimension $d$. This is the only result presented in this paper in which the exponent does not deteriorate when the dimension $d$ gets larger. We expect that a similar dimension-independent estimate also holds for $p$-optimal transport whenever $q>p>1$, see Remark \ref{r:qcostnotdone}, but we will not  establish this here. Instead, combining the stability results of Mischler-Trevisan \cite{mischlertrevisan} and a truncation argument, we prove that the map $\Phi_{p,\rho}^\LOT$ is bi-Hölder on $\Mom_{q,R}$ as soon as $q>p>1$. The Hölder exponent obtained by this argument, however, deteriorates as the dimension grows.

\begin{theorem}\label{t:LOTq}
Let $\rho \in\Prob_{\mathrm{ac}}(\Rsp^d)$ and assume that its density is bounded from above and below by positive constants on its support, which is compact and convex. Let $1<p<q$ and $R\geq 0$. Then there is an exponent $\alpha = \alpha(p,q,d) > 0$ and a constant $C=\cst(p,q,R)$ such that
$$ 
\forall \mu,\nu \in\mathcal{M}_{q,R}, \qquad W_p(\mu,\nu)\leq \nr{T_{\rho\to\mu,p} - T_{\rho\to\nu,p}}_{\LL^p(\rho)} \leq C\Wass_1(\mu,\nu)^{\alpha}.
$$
As a consequence, the mapping $\Phi^\LOT_{p,\rho}$ is a bi-Hölder embedding from $\mathcal{M}_{q,R}$ into $\LL^p(\rho,\Rsp^d)$.
\end{theorem} 

The following proposition shows that the strict inequality $q>p$ in Theorem \ref{t:LOTq} and  the hypothesis $q>2$ in \Cref{t:p>2} are both sharp.

\begin{proposition}\label{p:turningpointp}
Let $p \in (1,+\infty)$, let $R>0$, and let $\rho$ be uniform on the unit disk in $\Rsp^2$. Then there exist two sequences $(\mu_n)_{n\in\N}$ and $(\nu_n)_{n\in\N}$ of elements of $\mathcal{M}_{p,R}$ such that 
$$
\lim_{n\rightarrow +\infty} W_p(\mu_n,\nu_n)=0, \qquad\liminf_{n\to+\infty}\|T_{\rho\to\mu_n,p}-T_{\rho\to\nu_n,p}\|_{\LL^p(\rho)} > 0.
$$
As a consequence, the mapping $\mu\mapsto T_{\rho\to\mu,p}\in \LL^p(\rho,\Rsp^d)$ is not uniformly continuous (and a fortiori not Hölder) on $\mathcal{M}_{p,R}$.
\end{proposition}

To complete our study of the LOT embedding, we prove a version of Theorem \ref{t:p>2} that holds for any relatively compact subset $\Mom \subseteq \Prob_2(\Rsp^d)$, from which Theorem~\ref{t:p>2} will actually be deduced. Since $\Phi_{2,\rho}^\LOT:\mu \in\Mom \subseteq \Prob_2(\Rsp^d) \mapsto T_{\rho\to \mu,2}$  is continuous, Heine's theorem already implies that its restriction to a relatively compact set is uniformly continuous. 
In the next theorem, we make this uniform continuity quantitative by establishing an explicit modulus depending only  on a uniform control on the quadratic tails of the measures in $\Mom$. To this end, given any subset $\Mom$ of $\Prob_2(\Rsp^d)$, we will denote 
\begin{equation}
\label{e:f(S)} f_\Mom(S) = \sup_{\mu\in \Mom} \int \abs{x}^2 \mathbf{1}_{\abs{x}>S} \dd\mu(x).
\end{equation}
With this notation, the set $\Mom$ is relatively compact if and only if $\lim_{S\to +\infty} f_\Mom(S) = 0$, as shown for instance in \cite[Proposition 2.2.3]{panaretos}.
\begin{theorem}\label{t:relativecompact} 
Let $\rho \in\Prob_{\mathrm{ac}}(\Rsp^d)$ and assume that its density is bounded from above and below by positive constants on its support, which is compact and convex. Let $\mathcal{M}$ be a subset of $(\mathcal{P}_2(\R^d),W_2)$, and define $f_{\Mom}$ as in \eqref{e:f(S)}. Then,
$$
\forall \mu,\nu\in\mathcal{M}, \qquad W_2(\mu,\nu)\leq \|T_{\rho\rightarrow \mu,2}-T_{\rho\rightarrow \nu,2}\|_{\LL^2(\rho)}\leq C\omega_\Mom(W_2(\mu,\nu)),
$$
where the constant $C\geq 0$ depends only on $\rho$ and $d$ and where $$\omega_\Mom(r) =  \sqrt{\inf_{S>0} S^{\frac{3}{2}} r^{\frac{1}{2}} + f_\Mom(S)}.$$ In particular, if the set $\Mom$ is relatively compact in $\Prob_2(\Rsp^d)$, then the mapping $\Phi^\LOT_{2,\rho}$ is uniformly continuous with continuity modulus $C\omega_\Mom$.
\end{theorem}


\subsection{Embedding via sliced Wasserstein distances} 
We now turn to the second type of embedding, which is directly derived from the sliced-Wasserstein distance. The $p$-sliced-Wasserstein distance was introduced in \cite{rabinpeyredelonbernot} as a computationally cheaper alternative to the standard Wasserstein distance. It can be defined by averaging $p$th powers of the $p$-Wasserstein distances between projections of the two measures on lines. Given a direction $\theta\in\Sph^{d-1}$, we let $P_\theta: x\in\Rsp^d\mapsto \sca{x}{\theta} \in\Rsp$ where $\sca{x}{\theta}$ denotes the Euclidean scalar product, and we set
$$\SlicedWass_p^p(\mu,\nu) = \int_{\Sph^{d-1}} \Wass_p^p(P_{\theta\#}\mu, P_{\theta\#}\nu)\dd\theta,$$
where $\dd\theta$ is the uniform probability measure on the sphere $\Sph^{d-1}$. Since the $p$-Wasserstein distance
between two measures on the real line equals the $\LL^p$ distance between their quantile functions, it follows that $\SlicedWass_p(\mu,\nu) = \nr{\Phi_p^\SW(\mu) - \Phi_p^\SW(\nu)}_{\LL^p(\Sph^{d-1} \times [0,1])}$, where the \emph{sliced-Wasserstein embedding} is defined by
\begin{equation}\label{eq:sw-embedding} \begin{aligned}
    \Phi_p^\SW : \Prob_p(\Rsp^d) \to \LL^p(\Sph^{d-1}\times[0,1])\\
\qquad \mu \mapsto (\theta,u) \mapsto Q_{P_{\theta\#}\mu}(u).
\end{aligned} 
\end{equation}
Since the projections $P_\theta$ are $1$-Lipschitz, one has $\Wass_p(P_{\theta\#}\mu, P_{\theta\#}\nu) \leq \Wass_p(\mu,\nu)$. Integrating this inequality then shows that $\Phi^\SW_p$ does not increase distances, i.e.
$$ \SlicedWass_p(\mu,\nu) = \nr{\Phi_p^\SW(\mu) - \Phi_p^\SW(\nu)}_{\LL^p(\Sph^{d-1} \times [0,1])} \leq \Wass_p(\mu,\nu).$$
The reverse inequality, showing that the Wasserstein distance can (sometimes) be controlled by a power of the sliced-Wasserstein distance, or equivalently that $\Phi_p^\SW$ is reverse-Hölder, is a nontrivial fact first established by Bonnotte (see \Cref{th:bonnotte}). The corresponding qualitative statement is much older: by the Cram\'er-Wold theorem, a probability measure is uniquely determined by its one-dimensional projections, which implies that $\Phi_p^\SW$ is injective and that $\SlicedWass_p$ is a distance. Quantitative versions of the Cram\'er-Wold theorem have been the subject of recent work by Bobkov and Götze \cite{bobkov}. Carlier, Figalli, Mérigot and Wang have established that $\Phi_p^\SW$ is bi-Hölder in the compact setting, i.e. on $\Mom_{\infty,R} = \Prob(B(0,R))$, but also when the measures have bounded exponential moments (Corollary 2.2 in \cite{carlierfigallimerigotwang}). Here, we obtain a similar bi-Hölder estimate  for $\Phi_p^\SW$ but on the much larger class $\Mom_{q,R}$, for any $q>p$.

\begin{theorem}\label{t:compareWpSWpholder} Let $q>p\geq 1$ and let $R\geq 0$. Then, the $p$-sliced-Wasserstein embedding is bi-Hölder on $\Mom_{q,R}$. More precisely, it satisfies for some
constant $C = \cst(d,p,q,R)>0$,
$$ C \Wass_p(\mu,\nu)^{\alpha} \leq \nr{\Phi^\SW_p(\mu) - \Phi^\SW_p(\nu)}_{\LL^p(\Sph^{d-1}\times[0,1])} \leq \Wass_p(\mu,\nu)$$
where $\alpha=\frac{p(dq-1)}{q-p}$ if $d$ is odd and $\alpha =
   \frac{p((d+1)q-1)}{q-p}$ if not.
\end{theorem}
The dependence of the exponent on the left-hand side on the ambient dimension is unavoidable, reflecting the fact that the sliced-Wasserstein distance is much weaker than the standard Wasserstein distance. In the compact case, this is explained in \S4.2 in \cite{merigot2026embedding}. We also prove the following theorem, which may be of independent interest.
\begin{theorem} \label{t:sw}
Let $q>1$, $R>0$, and let $p,r\in [1,q)$. Then, the distances $\SlicedWass_p$, $\SlicedWass_r$, $\Wass_p$ and $\Wass_r$ are all bi-Hölder equivalent on $\Mom_{q,R}$. 
\end{theorem}
These positive results are complemented by a counterexample showing that the inverse of $\Phi^\SW_p$ is not uniformly continuous on $\Phi^\SW_p(\Mom_{p,R})$ for any integer $p \geq 1$, when $d\geq 2$  and $R>0$. Equivalently, the $\SlicedWass_p$ and $\Wass_p$ distances are \emph{not} uniformly equivalent in this case. Our proof techniques do not handle the case of non-integer $p\geq 1$. The corresponding result for non-integer $p\geq 1$ remains open. 

\begin{proposition}[\Cref{th:nonequiv-sw}]
Let $p\geq 1$ be an integer and let $d\geq 2$. Then, there exist probability measures $(\mu_n)_{n\geq 0}$ and $(\nu_{n})_{n\geq 0}$ on $\Rsp^d$ such that $M_p(\mu_n) = M_p(\nu_n) = 1$ and 
$$\liminf_{n\to +\infty} \Wass_p(\mu_n,\nu_n) > 0, \quad \lim_{n\to+\infty}\SlicedWass_p(\mu_n,\nu_n) = \lim_{n\to+\infty} \nr{\Phi_p^\SW(\mu_n) - \Phi_p^\SW(\nu_n)}_{\LL^p(\Sph^{d-1}\times[0,1])} = 0.$$  
\end{proposition}


\subsection{Kernel mean embeddings} We finally turn to the third embedding considered in this article.
Let $K:\R^d\times\R^d\rightarrow \R$ be a symmetric positive semi-definite  kernel.
The reproducing kernel Hilbert space associated to $K$ by the Moore-Aronszajn theorem is denoted  by $\mathcal{H}_K$, and its scalar product by $\langle \cdot,\cdot\rangle_{\mathcal{H}_K}$. Recall that this means that for any $f\in\mathcal{H}_K$ we have $f(x)=\langle f,K(x,\cdot)\rangle_{\mathcal{H}_K}$.
The associated Kernel Mean Embedding (KME) is the linear mapping $\mu\mapsto K_\mu\in\mathcal{H}_K$ where $\mu$ belongs to the set $\mathcal{S}_K$ of signed finite measures such that $\int_{\R^d}\sqrt{K(x,x)}\dd|\mu|(x)<+\infty$, and $K_\mu$ satisfies 
$$
\forall f\in \mathcal{H}_K,\qquad \int_{\R^d}f\dd\mu=\langle f,K_\mu\rangle_{\mathcal{H}_K}.
$$
The distance $d_K(\mu,\nu)$ on $\mathcal{S}_K$ defined by
$$
d_K(\mu,\nu)=\|K_\mu-K_\nu\|_{\mathcal{H}_K}
$$
is called Maximum Mean Discrepancy (MMD), see \cite{gretton1,gretton2}.

Here we focus on the particular case where $K=\kappa_\beta$ with 
\begin{equation}\label{e:kalpha}
\kappa_\beta(x,y)=|x|^{2\beta}+|y|^{2\beta}-|x-y|^{2\beta}
\end{equation}
where $\beta\in(0,1)$. Notice that $\int_{\R^d}\sqrt{\kappa_\beta(x,x)}\dd\mu(x)=\sqrt{2}\int_{\R^d}|x|^\beta \dd\mu(x)$, hence $\mathcal{M}_1\subset\mathcal{S}_{\kappa,\beta}$. The MMDs associated to the kernels $\kappa_\beta$ are the so-called energy distances, first introduced by Székely (see \cite{szekely}, \cite{sejdinovic}), and equivalently defined via the Riesz kernels $-|x-y|^\beta$ (but \eqref{e:kalpha} has the advantage of being positive semi-definite, while Riesz kernels are only conditionally positive).  The reason we focus on the specific kernels $\kappa_\beta$ is that the associated kernel mean embedding compares particularly well to the $W_q$ distances. Indeed, the following result follows essentially from \cite[Proposition 17 and Lemma 20]{modeste} (see Section \ref{s:mmd} for details): 

\begin{theorem}\label{t:MMDW1}
Let $\beta\in(0,1)$ and $p\geq 1$. Consider the kernel mean embedding $\mu\mapsto K_\mu$ associated to the kernel $\kappa_\beta$. Then for any $R>0$ and $q>p$, there exists $C>0$ such that
\begin{equation}\label{e:kalphasmallerW1}
\forall \mu,\nu\in\mathcal{M}_{1,R}, \qquad \|K_\mu-K_\nu\|_{\mathcal{H}_{\kappa_\beta}}\leq  CW_1(\mu,\nu)^{\frac{\beta}{2}}
\end{equation}
and 
\begin{equation}\label{e:W1smallerkalpha}
\forall \mu,\nu\in\mathcal{M}_{q,R}, \qquad W_p(\mu,\nu)\leq C\|K_\mu-K_\nu\|_{\mathcal{H}_{\kappa_\beta}}^\delta
\end{equation}
where $\delta=\frac{q-p}{pq(d+\beta+1)-p\beta}$.
As a consequence, the mapping $\mu\mapsto K_\mu$ is a bi-Hölder embedding of $\mathcal{M}_{q,R}$ into the Hilbert space $\mathcal{H}_{\kappa_\beta}$.
\end{theorem}
The exponent in this result depends (badly) on the dimension $d$, and this cannot be avoided: when $\nu=\mu_n$ is an $n$-point random empirical approximation of the uniform probability density $\mu$ over a nice bounded open set, then $W_q(\mu_n,\mu)\asymp n^{-1/d}$, whereas $\|K_{\mu_n}-K_\mu\|_{\mathcal{H}_{\kappa_\beta}}\asymp n^{-\frac12}$, see for instance \cite[Eq. (1.7)]{colasanto} (notice that in this paper the kernel is $-|x-y|^\beta$, but this yields the same MMD distance). See \cite[Lemmas 9 and 41]{vayer} for an argument working for any bounded kernel.

Finally, we observe that the above kernel mean embedding is not uniform over the set of measures with bounded $p$-th moment:
\begin{proposition}\label{p:nouniformMMD}
Let $\beta\in (0,1)$, $R>0$ and $p\geq 1$. Then there exist $\mu\in\mathcal{M}_{p,R}$ and $\mu_n\in\mathcal{M}_{p,R}$ for $n\in\N$ such that 
$$
\liminf_{n\rightarrow +\infty}W_p(\mu_n,\mu)>0, \qquad \lim_{n\rightarrow +\infty}\|K_{\mu_n}-K_\mu\|_{\mathcal{H}_{\kappa_\beta}}=0.
$$
As a consequence, the kernel mean embedding $\mathcal{M}_p\ni \mu\mapsto K_\mu$ associated to the kernel $\kappa_\beta$ is not uniform over $\mathcal{M}_{p,R}$. 
\end{proposition}

Let us conclude with a few bibliographical remarks. For a result related to our Theorem \ref{t:MMDW1} in the case $\beta=1/2$, see also \cite[Theorems 6 and 7]{toscani}. Note however that \cite[Theorem 7]{toscani} has a condition on second moments, not on first moments. Theorem 2 in \cite{hertrich} is also about the case $\beta=1/2$,
but it only applies to measures supported over a ball. Theorems 14 and 15 in \cite{vayer} provide comparisons between $W_q$ and MMD distances, but only for a class of probability densities which are regular enough. We also mention the paper \cite{carrillo}, notably Section 2.4, which provides comparison of $W_1$ with many other norms, but no Hilbertian norm. The survey \cite{toscani} presents thoroughly Wasserstein metrics, Fourier based metrics, and energy distances, as well as some of their relations.

\subsection*{Acknowledgments} We thank Titouan Vayer for an interesting discussion regarding MMD distances. C.L. acknowledges funding from PEPS JCJC. The authors
acknowledge the support of the Agence nationale de la recherche, through the PEPR PDE-AI project (ANR-23-PEIA-0004).

\subsection*{AI disclosure} The only use of AI tools in the preparation of this paper was in the search for the counterexample given in Section \ref{s:nonequivalence}. All other content, as well as the writing itself, is due to the authors, who take full responsibility for the content of the article.

\section{Proofs for embedding via linearized optimal transport}

In this section, we prove successively Theorems \ref{t:relativecompact}, \ref{t:p>2} and \ref{t:LOTq}, and Proposition \ref{p:turningpointp}.

\subsection{Proof of \Cref{t:relativecompact}.} Let $\mu_0,\mu_1\in \Mom$ and $T_i = T_{\rho\to\mu_i, 2}$ be the quadratic optimal transport map from $\rho$ to $\mu_i$. 
Given a radius $S>0$, we denote $E_{S,i} = \{ x \in\Rsp^d \mid \abs{T_i(x)}\leq S\}$ and $E_S = E_{S,0}\cap E_{S,1}$. We also denote $\rho_S = \mathbf{1}_{E_S} \rho$ the restriction of $\rho$ to the set $E_S$ and $\rho_S^\perp = \rho - \rho_S$ its complement. The squared distance between $T_0$ and $T_1$ can be decomposed as:
$$ \nr{T_0-T_1}_{\LL^2(\rho)}^2 = \nr{T_0-T_1}_{\LL^2(\rho_S)}^2+\nr{T_0-T_1}_{\LL^2(\rho_S^\perp)}^2 =: A+B.$$
The term $B$ in this sum can be controlled using the definition of $f_\Mom$ and the fact that $\mu_i = T_{i\#}\rho$:
\begin{align*} 
B = \nr{T_0-T_1}_{\LL^2(\rho_S^\perp)}^2
&\leq 2 \sum_{i\in\{0,1\}} \nr{T_i}^2_{\LL^2(\rho_S^\perp)} \\
&\leq 2  \sum_{i\in\{0,1\}} \int_{E_{S,i}^c} \abs{T_i(x)}^2 \dd\rho(x) \\
&= 2 \sum_{i\in\{0,1\}} \int_{T_i(E_{S,i}^c)} \abs{y}^2 \dd\mu_i(y) 
\leq 2 f_\Mom(S)
\end{align*}
where $E_{S,i}^c=\R^d\setminus E_{S,i}$.
The upper bound on the term $A$ relies on stability results from \cite{merigot}. By Brenier's theorem, for every $i\in\{0,1\}$ there exists a lower semi-continuous convex functions $\phi_i: \Rsp^d\to \Rsp\cup\{+\infty\}$ such that $T_i = \nabla \phi_i$. We may also assume that $\phi_i \equiv +\infty$ outside of $\spt(\rho)$. We then define $\psi_i = \phi_i^*$, which implies that $\psi_i$ is $L$-Lipschitz with $L=\max_{x\in\spt(\rho)}\abs{x}$. Define $\psi_{i,S} = \left.\psi_{i}\right\vert_S$ the restriction of $\psi_i$ to $B(0,S)$, extended by $+\infty$ outside of $B(0,S)$. The Fenchel-Young equality implies that on $E_{S,i},$ $\psi_{i,S}^* = \psi_i$ and $\nabla \psi_{i,S}^* = \nabla\psi_i^*$. By \cite[Theorem 2.2]{merigot}, there is a constant $C$ depending on $d$ and $\rho$ (but not $S$) such that 
$$ A^2 = \Vert{T_{\mu_0} - T_{\mu_1}}\Vert^4_{L^2(\rho_S)} \leq CS^3 B_{\rho_S}(\psi_0,\psi_1),$$
where $B_\rho$ is the Bregman-Kantorovich divergence with respect to $\rho_S$. We recall that thanks to Lemma A.1 in \cite{merigot}, this divergence can be written using the Fenchel-Young loss $G_{\phi}(x,y) := \phi(x) + \phi^*(y) - \langle{x}\vert{y}\rangle$ :  
\begin{equation}\label{e:bregmankantofenchel}
\begin{aligned}
B_{\rho_S}(\psi_{0,S},\psi_{1,S}) &= \int_{E_S} G_{\psi_{1,S}^*}(x, \nabla \psi^*_{0,S}(x)) d\rho_S(x)\\
&= \int_{E_S} G_{\phi_{1}}(x, \nabla \phi_0(x)) d\rho_S(x)\\
&\leq \int G_{\phi_1}(x,\nabla\phi_0(x) \dd\rho(x),
\end{aligned}
\end{equation}
where the last inequality uses the non-negativity of $G_{\phi_1}$. Finally, using that the functions $\psi_i$ are $L$-Lipschitz, and Kantorovich-Rubinstein duality, we obtain
$$ B_{\rho}(\psi_0,\psi_1)+B_\rho(\psi_1,\psi_0) = \int (\psi_0 -
\psi_1)d(\mu_1 - \mu_0) \leq C W_1(\mu_0,\mu_1),$$
where $C$ still depends only on $\rho$ and $d$. Since both Bregman-Kantorovich divergences are non-negative, we finally obtain
$$A \leq CS^{\frac{3}{2}} \Wass_1(\mu_0,\nu_1)^{\frac{1}{2}}$$ 
Summing the bounds on $A$ and $B$, we obtain
$$ \nr{T_0-T_1}_{\LL^2(\rho)}^2 \leq C(S^{\frac{3}{2}} \Wass_1(\mu_0,\mu_1)^{\frac{1}{2}} + f_\Mom(S)).$$
Since the above inequality holds for any $S>0$, we may take an infimum with respect to $S>0$ in the right-hand side and taking a square root we recover the definition of $\omega_\Mom$.\\
Finally, if $\Mom$ is relatively compact in $\Prob_2(\Rsp^d)$, then $\lim_{S\to +\infty} f_\Mom(S) = 0$. Thus, for any $\varepsilon>0$, there exists $S_0$ such that $f_\Mom(S_0) \leq \varepsilon^2$. Then,
$$ \lim\sup_{r\to 0} \omega_\Mom(r) \leq C \sqrt{\limsup_{r\to 0} S_0^{\frac32} r^{\frac12} + f_\Mom(S_0)} \leq \varepsilon.$$
Since the above inequality holds for all $\varepsilon\geq 0$, we obtain that $\lim_{r\to 0}\omega_\Mom(r) = 0$, establishing that $\omega_\Mom$ is a valid continuity modulus.

\subsection{Proof of Theorem \ref{t:p>2}} Let $\mu \in \Mom:=\Mom_{q,R}$ for $q>2$. Then, for every $S>0$,
\begin{equation*} 
\int_{B(0,S)^c} \abs{x}^2 \dd \mu(x) = \int_{B(0,S)^c} \abs{x}^{q+(2-q)} \dd \mu(x)  \leq S^{2-q} M_q(\mu) \leq S^{2-q} R^q.
\end{equation*}
Since this holds for any $\mu\in\Mom$, we obtain $f_\Mom(S) \leq S^{2-q} R^q$. Then, 
$$ \omega_\Mom(r) = C \sqrt{\inf_{S>0} S^{3/2} r^{1/2} + f_\Mom(S)} \leq C \sqrt{\inf_{S>0} S^{3/2} r^{1/2}+ S^{2-q} R^q}.$$
Choosing $S$ so that the two terms of the sum have the same order, i.e. $S = (R^{2q}/r)^{1/(2q-1)}$  gives the following explicit continuity modulus
\begin{equation*}
    \omega_\Mom(r) \leq C R^{\frac{3q}{2(2q-1)}} r^{\frac{q-2}{2(2q-1)}}
\end{equation*} 
which proves \eqref{e:2q-1}.

Regarding the sharpness of the exponent in \eqref{e:2q-1}, we proceed as follows. 
Let $\rho = \mathrm{Leb}_{[0,1]^2}$, let $q>2$ and for some parameters $r_n, m_n, \theta_n > 0$ define
$$ \mu_n = (1-m_n)\delta_0 + m_n \delta_{(r_n,1/2)}, \quad \nu_n = R_{\theta_n\#} \mu_n.$$
We choose $m_n$ such that $M_q(\mu_n) = m_n r_n^q = 1$, i.e. $m_n = r_n^{-q}$. The $1$-Wasserstein distance between $\mu_n$ and $\nu_n$ can be bounded using the rotation $R_\theta$ as a transport map, giving
$$ \Wass_1(\mu_n,\nu_n) \leq m_n \cdot (\theta_n r_n) = \theta_n r_n^{1-q}$$
The quadratic optimal transport plan $T_{\mu_n}$ between $\rho$ and $\mu_n$ maps $[0,1-m_n]\times [0,1]$ to $0$ and $[1-m_n,1]\times[0,1]$ to $(r_n,1/2)$. For $\nu_n$, if $\theta_n \leq m_n$, the optimal transport map has a similar structure, but the vertical segment $S_n := \{m_n\}\times [0,1]$ is also rotated by $\theta_n$ around the point $(m_n,\frac{1}{2})$. More precisely,
$$
T_{\nu_n}(x)=\begin{cases} 0 &\text{ if } \sca{x}{e_n}\leq \sca{(m_n,\frac12)}{e_n}\\ 
(r_n,1/2) &\text{ if } \sca{x}{e_n}>\sca{(m_n,\frac12)}{e_n}\end{cases}
$$
where $e_n=(\cos\theta_n,\sin\theta_n)$.
The $\LL^2$ distance between these maps is then 
$$ \nr{T_{\mu_n} - T_{\nu_n}}_{\LL^2(\rho)} \geq (\theta_n r_n^2)^{1/2} = \theta_n^{1/2} r_n.$$
Taking $\theta_n = m_n = r_n^{-q}$ in the previous estimates, we get
$$\nr{T_{\mu_n} - T_{\nu_n}}_{\LL^2(\rho)} \geq r_n^{1-\frac{q}{2}} \geq \Wass_1(\mu_n,\nu_n)^{\frac{\frac{q}{2}-1}{2q-1}} = \Wass_1(\mu_n,\nu_n)^\frac{q-2}{2(2q-1)} = \Wass_1(\mu_n,\nu_n)^{\frac{1}{4} - \frac{3}{4(2q-1)}},$$
thus showing that the exponent in \eqref{e:2q-1} is optimal.

\begin{remark}\label{r:qcostnotdone}
It is likely that Theorem \ref{t:p>2} generalizes to optimal transport with respect to $p$-cost on $\R^d$ for $p>1$: one would obtain for $q>p$ and $R>0$ a quantitative stability estimate, uniform over $\mathcal{M}_{q,R}$, and with dimension-free exponent (depending only on $p,q$). But proving this would probably require first generalizing \cite{merigot} to $p$-costs, which is not yet done.
\end{remark}

\subsection{Proof of Theorem \ref{t:LOTq}}\label{s:proofLOTq}

Let $\rho$ be a probability density with bounded and convex support $\Omega$. Assume $\rho$ is bounded above and below on $\Omega$. Denote 
\begin{equation}\label{e:Omegaeta}
\Omega_\eta = \{x\in\Omega\mid \dd(x,\partial\Omega)\geq \eta\}.
\end{equation}
Given  $\mu\in\Prob_p(\Rsp^d)$, we denote by $\phi_{\rho\to\mu,p}$ the $c_p$-concave Kantorovich potential on the side of $\rho$ that has zero $\rho$-mean and  $\psi_{\rho\to\mu, p}$ the $c_p$-transform of $\phi_{\rho\to\mu,p}$, defined by 
\begin{equation}\label{e:cptransform}
\psi_{\rho\to\mu, p}(y)=\inf_{x\in \Omega} |x-y|^p-\phi_{\rho\rightarrow\mu, p}(y).
\end{equation}

\begin{lemma}\label{l:boundedtransporteta}
Let $q>p>1$. There exists $C>0$ such that for any $\eta>0$ and $R\geq 1$,
$$\forall \mu\in\mathcal{M}_{q,R},\qquad \|T_{\rho\rightarrow \mu,p}\|_{L^\infty(\Omega_\eta)}\leq C R^{\frac{p+1}{p-1}}\eta^{-\frac{d}{q}\frac{p+1}{p-1}-\frac{1}{p-1}}.$$
\end{lemma}
\begin{proof}
Let $T=T_{\rho\rightarrow \mu,p}$. Let $x\in\Omega$, let $\eta={\rm dist}(x,\partial\Omega)$ and $\mathcal{C}_x$, the intersection of a small annulus with a cone of apex $x$:
$$
\mathcal{C}_x=\left\{y\in B(x,\eta)\setminus B(x,\eta/2)\mid (y-x)\cdot (T(x)-x)\geq \frac12 |y-x||T(x)-x|\right\}\subset \Omega.
$$ 
Let $A=1+\sup_{x\in \Omega}|x|$ and let 
$$
B:=|T(x)|A^{-\frac{p}{p+1}}\eta^{\frac{1}{p-1}}.
$$
We will assume below that $B$ is large (depending on $A,p$). 
Let 
$$
\alpha= T(x)-x, \qquad \beta=T(y)-T(x), \qquad \gamma=y-x.
$$
Cyclical monotonicity yields
$$
|\alpha|^p-|\alpha-\gamma|^p\leq |\alpha+\beta|^p - |\alpha+\beta-\gamma|^p.
$$
We have 
\begin{align*}
|\alpha|^p-|\alpha-\gamma|^p=|\alpha|^p-|\alpha|^p\left(1-2\frac{\alpha}{|\alpha|^2}\cdot \gamma +\frac{|\gamma|^2}{|\alpha|^2}\right)^{p/2}\geq  \frac{p}{2}|\alpha|^{p-2} \alpha\cdot \gamma.
\end{align*}
for $|\alpha|$ large enough (which is the case for $B$ large enough).
We deduce that for any $y\in\mathcal{C}_x$ there holds
$
|\alpha+\beta|^p \geq \frac{p}{2}|\alpha|^{p-1}|\gamma|
$
hence $$|T(y)-x|^q\geq \frac{p}{2}|T(x)-x|^{p-1}|y-x|.$$
If $B$ is large enough then $|T(x)-x|\geq \frac12|T(x)|$ and we get for any $y\in\mathcal{C}_x$
\begin{equation}\label{e:Ty}
|T(y)|\geq \frac{p^{\frac1p}}{2}|T(x)|^{\frac{p-1}{p}}\eta^{\frac1p}-|x|\geq  \frac{p^{\frac1p}}{2} B^{\frac{p-1}{p}}A^{\frac{p-1}{p+1}}-A\geq  \frac{p^{\frac1p}}{4}  B^{\frac{p-1}{p}} A^{\frac{p-1}{p+1}}.
\end{equation}
Since $\rho(C)\geq c\eta^d$ for some universal $c>0$, we deduce that $c\eta^d A^{q\frac{p-1}{p+1}}B^{q\frac{p-1}{p+1}}\leq R^q$. We conclude that $B\lesssim R^{\frac{p+1}{p-1}}\eta^{-\frac{d}{q}\frac{p+1}{p-1}}$ and finally $|T(x)|\lesssim R^{\frac{p+1}{p-1}}\eta^{-\frac{d}{q}\frac{p+1}{p-1}-\frac{1}{p-1}}$ where the hidden constants depend only on $A,q,p$.
\end{proof}

\begin{lemma}\label{l:boundW1-qmoment}
Let $q>p>1$ and $R\geq1$. Then there exists a constant $C=C(p,q,\rho)>0$ such that, for every $\mu,\nu\in\Mom_{q,R}$,
$$\int (\psi_{\rho\to\mu,p} - \psi_{\rho\to\nu,p})\, \dd(\mu-\nu)\leq C R^{\frac{q(p-1)}{q-1}} \Wass_1(\mu,\nu)^{\frac{q-p}{q-1}}$$
\end{lemma}

\begin{proof}
    We first establish a Lipschitz-like behaviour for the potential $\psi_{\rho\to\mu,p} = \phi_{\rho\to\mu,p}^c$. 
    Let $x \in \Omega$. Then, for any $y,z\in\Rsp^d$ we have 
    $$ \abs{x-y}^p - \abs{x-z}^p \leq  p \abs{y-z} \max\big(\abs{x-y},\abs{x-z}\big)^{p-1} \leq C \abs{y-z}(1+\abs{y}+\abs{z})^{p-1}$$
    for some constant $C = C(p,\Omega)$. Thus, 
    \begin{align*}
        \psi_{\rho\to\mu,p}(z) 
        &= \inf_{x\in\Omega} \abs{x-z}^p - \psi_{\rho\to\mu,p}(x)\\
        &\leq \inf_{x\in\Omega} \abs{x-y}^p - \psi_{\rho\to\mu,p}(x) + C \abs{y-z}(1+\abs{y}+\abs{z})^{p-1}\\
        &= \psi_{\rho\to\mu,p}(y) + C \abs{y-z}(1+\abs{y}+\abs{z})^{p-1}.
    \end{align*}
    Let $\gamma \in\Gamma(\mu,\nu)$ be optimal for the $1$-Wasserstein distance. Then,
    \begin{align*}
        \int\psi_{\rho\to\mu,p} \dd(\mu-\nu) &= \int (\psi_{\rho\to\mu,p}(y) - \psi_{\rho\to\nu,p}(z))\, \dd\gamma(y,z) 
        &\leq C \int \abs{y-z}(1+\abs{y}+\abs{z})^{p-1} \dd\pi(y,z).
    \end{align*}
    Setting $\alpha  = \frac{q-p}{q-1}\in(0,1)$ and using Hölder's inequality, we have 
    \begin{align*}
        \int \abs{y-z}(1+\abs{y}+\abs{z})^{p-1} \dd\pi(y,z) 
        &\leq \int \abs{y-z}^\alpha (1+\abs{y}+\abs{z})^{p-\alpha}\dd\pi(y,z)\\
        &\leq  \left(\int \abs{y-z}\dd\pi(y,z)\right)^\alpha\left(\int(1+\abs{y}+\abs{z})^{\frac{p-\alpha}{1-\alpha}}\dd\pi(y,z)\right)^{1-\alpha}\\
        &\leq C \Wass_1(\mu,\nu)^{\frac{q-p}{q-1}} \left(\int(1+\abs{y}+\abs{z})^{q}\dd\pi(y,z)\right)^{\frac{p-1}{q-1}}\\
        &\leq C R^{\frac{q(p-1)}{q-1}} \Wass_1(\mu,\nu)^{\frac{q-p}{q-1}} 
    \end{align*}
    where we used $1-\alpha = \frac{p-1}{q-1}$ and $\frac{p-\alpha}{1-\alpha} = q$, the assumption $M_q(\mu)\leq R^q$, $M_q(\nu)\leq R^q$ and finally $R\geq 1$.
\end{proof}

\begin{proof}[Proof of Theorem \ref{t:LOTq}]

We write $T_\mu = T_{\rho\rightarrow \mu,p}$ and $T_\nu=T_{\rho\rightarrow \nu,p}$ for simplicity. Let $\eta>0$ and define $\rho_\eta = \rho\vert_{\Omega_\eta}$, $\mu_\eta = T_{\mu\#}\rho_\eta$ and $\nu_\eta=T_{\nu\#}\rho_\eta$. By Lemma \ref{l:boundedtransporteta}, $\mu_\eta$ and $\nu_\eta$ are supported in a ball of radius  $R_\eta = CR^{\frac{p+1}{p-1}}\eta^{-\frac{d}{q}\frac{p+1}{p-1}-\frac{1}{p-1}}$. In the sequel, the constant $c$ depends only on $p,q,R,$ and $A=1+\sup_{x\in \Omega}|x|$ and may change from line to line.

Since $\rho_\eta$ is bounded above and below on $\Omega_\eta$ which is convex, we show in the appendix how a modification of \cite[Theorem 1.2]{mischlertrevisan} yields the existence of $\delta>0$ (depending on $p$ only) and $c>0$ (depending on $A,p,q$ only) such that 
\begin{equation}\label{e:mischlertrevisan}
\nr{T_\nu-T_\mu}_{\LL^2(\rho_\eta)} \leq c R_\eta^{1-(p-1)\delta} \left(\int (\psi_{\rho\to\mu,p} - \psi_{\rho\to\nu,p}) \dd(\nu-\mu)\right)^\delta.
\end{equation}
Therefore,
\begin{align}
\|T_\nu-T_\mu\|_{L^p(\rho_\eta)}^p\leq \|T_\nu-T_\mu\|_{L^1(\rho_\eta)}R_\eta^{p-1}&\leq \|T_\nu-T_\mu\|_{L^2(\rho_\eta)}R_\eta^{p-1}\nonumber \\&\leq c R_\eta^{p-1}R_\eta^{1-(p-1)\delta}
\left(\int (\psi_{\rho\to\mu,p} - \psi_{\rho\to\nu,p}) \dd(\nu-\mu)\right)^\delta \nonumber\\ &\leq c R_\eta^{p+\delta\frac{p-1}{q-1}} \Wass_1(\mu,\nu)^{\delta \frac{q-p}{q-1}} \label{e:TmuTnuLq}
\end{align}
The second inequality follows from Cauchy-Schwarz, the third inequality from Corollary \ref{c:localizedMT}, and the last inequality from Lemma \ref{l:boundW1-qmoment}.

Let $\sigma_\eta=  \rho - \rho_\eta$.
Using Hölder's inequality and then \cite[Proposition 4.6]{delalandemerigot} we get
\begin{equation}\label{e:TmuTnu+}\|T_\mu\|^p_{L^p(\sigma_\eta)}\leq M_r(\mu)^{\frac{p}{q}}\sigma_\eta(\R^d)^{1-\frac{p}{q}}\leq c\eta^{1-\frac{p}{q}}.
\end{equation}
Summing \eqref{e:TmuTnuLq} and \eqref{e:TmuTnu+}, we find for  $\alpha'=\left(p+\delta\frac{p-1}{q-1}\right)\left(\frac{d}{q}\frac{p+1}{p-1}+\frac{1}{p-1}\right)$
$$ \nr{T_\nu-T_\mu}^p_{\LL^p(\rho)} \leq c \eta^{-\alpha'} \Wass_1(\mu,\nu)^{\delta \frac{q-p}{q-1}} + c\eta^{1 - \frac{p}{q}}.$$
The exponent in the latter term is strictly positive if $q>p$. Hence, choosing $\eta$ is a way that $\eta^{-\alpha'} \Wass_1(\mu,\nu)^{\delta \frac{q-p}{q-1}} = \eta^{1 - \frac{p}{q}}$, we get the result.
\end{proof}

\subsection{Proof of Proposition \ref{p:turningpointp}}
For $n\in\N\setminus\{0\}$ and $\varepsilon\geq0$, let
$$
\mu_{n,\varepsilon}=\frac{1}{2n^p}\left(\delta_{(n,\varepsilon n)}+\delta_{(-n,-\varepsilon n)}\right)+\Bigl(1-\frac{1}{2n^{p}}\Bigr)\delta_0.
$$

These measures have moment of order $q$ uniformly bounded in $n\in\N$ and $\varepsilon\leq 1/100$. In the sequel, $n\geq 1$. We have
$$
W_p(\mu_{n,\varepsilon},\mu_{n,0})=(n^{-p} \varepsilon^pn^p)^{\frac1p}=\varepsilon.
$$
For notational simplicity, let $T_{n,\varepsilon}=T_{\rho\rightarrow \mu_{n,\varepsilon},p}$ denote  the optimal transport map for $p$-cost from $\rho$ to $\mu_{n,\varepsilon}$. We have
$$
T_{n,0}(x)=\begin{cases} 0 & \text{if } |x|\leq a_n \\ (-n,0) & \text{if } x\leq -a_n \\ (n,0) & \text{if } x\geq a_n\end{cases}
$$
where $0\leq a_n<1$ is defined by 
$$
\int_{a_n}^1 \sqrt{1-x^2} \dd x=\frac{1}{2n^p}.
$$
In particular one can check that $\frac{\sqrt{1-a_n^2}}{a_n}\leq 2n^{-p/3}$ for $n$ large enough.
Notice also that $T_{n,\varepsilon}$ is obtained by composing a rotation of angle $\tan^{-1}\varepsilon$ with $T_{n,0}$.
Therefore, as soon as $\varepsilon\geq 4n^{-p/3}$, 
$$
\|T_{n,\varepsilon}-T_{n,0}\|_{L^p(\rho)}=(n^{-p}n^p)^{\frac1p}=1.
$$
Taking $\mu_n=\mu_{n,0}$ and $\nu_n=\mu_{n,\varepsilon_n}$ with $\varepsilon_n=4n^{-p/3}$ and letting $n\rightarrow +\infty$, we get the result.

\section{Proofs for embedding via sliced-Wasserstein}

In this section, given a measure $\mu\in\Prob(\Rsp^d)$ and
$\theta\in\Sph^{d-1}$, we will denote $\mu_\theta = P_{\theta\#}\mu$
where $P_\theta(x) = \sca{x}{\theta}$. We recall that the uniform
probability measure over the unit sphere will be denoted
$\dd\theta$. Then, the sliced-Wasserstein embedding is
$$ \Phi_p^\SW : \Prob_p(\Rsp^d) \to \LL^p(\Sph^{d-1}\times[0,1]),
\qquad \Phi^\SW_p(\mu)(\theta,u) := Q^\mu_\theta(u), $$ and the
sliced-Wasserstein distance is $\SlicedWass_p(\mu,\nu) =
\nr{\Phi_p^\SW(\mu) - \Phi_p^\SW(\nu)}_{\LL^p(\Sph^{d-1}\times[0,1])}$.  In the statements and proofs
of this section, $c_d$ is a dimensional constant that may change from
line to line.

\subsection{Equivalence between $\SlicedWass_p$ and $\Wass_p$ on $\Mom_{q,R}$}

Proving that $\Phi^\SW_p$ is bi-Hölder over a class of measures is
equivalent to proving a bi-Hölder comparison between $\SlicedWass_p$
and $\Wass_p$ over this class. We start with the case $p=1$.

\begin{theorem} \label{t:comparW1sW1}
  Let $d\geq 2$ and $q>1$. Then, there exists a constant
  $c_{d,q}<\infty$ such that, for every $R\geq 0$ and
  \begin{equation}\label{eq:sliced-1-comparison}
    \forall \mu,\nu\in\mathcal{M}_{q,R}, \qquad \SlicedWass_1(\mu,\nu)
    \leq \Wass_1(\mu,\nu) \leq
    c_{d,q}R^{1-\beta}\SlicedWass_1(\mu,\nu)^\beta.
  \end{equation}
  where $\beta:=\frac{q-1}{dq-1}$ if $d$ is odd and
  $\beta:=\frac{q-1}{(d+1)q-1}$ if $d$ is even.
\end{theorem}

The proof of \Cref{t:comparW1sW1} relies on the following lemma, which
uses Radon transform arguments from
\cite{carlierfigallimerigotwang}. This lemma allows one to apply the
truncation of Kantorovich potentials originally found in \cite[Lemma
  3.1]{bobkov} in the case of the maximum sliced-Wasserstein distance.

\begin{lemma} \label{lem:carlierfigallimerigotwang-test}
Assume that $d\geq3$ is odd. There exists a dimensional constant $c_d$
such that for every $A\geq 0$, all $\mu,\nu\in\Prob_1(\Rsp^d)$ and
every Lipschitz $f:\Rsp^d\to\Rsp$ supported in the ball $B_A=B(0,A)$,
\[
  \abs{\int_{\Rsp^d}f\dd(\mu-\nu)} \leq
  c_d\Lip(f)A^{\frac{d-1}{d}}\SlicedWass_1(\mu,\nu)^{\frac1d}.
\]
\end{lemma}

\begin{proof}
Let $\chi$ be a smooth probability density supported in $B_1$, let
$\chi_\lambda=\lambda^{-d}\chi(\cdot/\lambda)$, and set
$\psi=f*\chi_\lambda$, with $\lambda> 0$. Since $f$ is Lipschitz,
$$\abs{\psi(x) - f(x)} \leq \sup_{y\in B(x,\lambda)} \abs{f(y) - f(x)} \leq \lambda \Lip(f).$$
Integrating this inequality gives us
\begin{equation}\label{eq:moll-error}
  \abs{\int(f-\psi)\dd(\mu-\nu)}\leq 2\Lip(f)\lambda.
\end{equation}
For odd $d$, the Radon inversion formula used in
\cite[(8),(12)]{carlierfigallimerigotwang} gives
\[
  \psi(x)=\int_{\Sph^{d-1}}\psi_\theta(\sca{x}{\theta})\dd\sigma(\theta),
  \qquad
  \psi_\theta(t)=c_d\,\partial_t^{d-1}R^\theta\psi(t), \qquad   R^\theta\psi(t):=\int_{\theta^\perp}\psi(x+t\theta)
  \dd\mathcal H^{d-1}(x)
  \]
Following \cite{carlierfigallimerigotwang}, we will estimate the
Lipschitz constant of $\psi_\theta$. Writing
$\partial_\theta:=\theta\cdot\nabla$ and using Fubini,
\begin{align}
  (\psi_\theta)'(t)
  &=c_d\int_{\theta^\perp}\partial_\theta^d\psi(x+t\theta)
    \dd\mathcal H^{d-1}(x)\notag\\
  &=\frac{c_d}{\lambda^{d-1}}
    \int_{B_A}\partial_\theta f(y) \int_{\theta^\perp}
    \,\partial_\theta^{d-1}\chi\!\left(
      \frac{x+t\theta-y}{\lambda}\right)\lambda^{-d}
    \dd\mathcal H^{d-1}(x) \dd y, \label{eq:psi-prime}
\end{align}
where we used $\spt(f) \subseteq B_A$ to restrict the first integral.
Moreover, writing $y = y' + \sca{y}{\theta}\theta$ with $y'\in\theta^\perp$ and substituting
$x = y' + \lambda w$, $w\in\theta^\perp$,
\begin{align*}
  \abs{\int_{\theta^\perp} \partial_\theta^{d-1}\chi\left(
    \frac{x+t\theta-y}{\lambda}\right)\lambda^{-d} \dd\mathcal H^{d-1}(x)}
  &= \frac{1}{\lambda} \abs{\int_{\theta^\perp}\partial_\theta^{d-1}\chi\!\left(w
     + \frac{t-\sca{y}{\theta}}{\lambda}\theta\right) \dd\mathcal H^{d-1}(w)} \\
  &\leq \frac{\omega_{d-1}}{\lambda} \nr{\partial_\theta^{d-1}\chi}_\infty
     \mathbf{1}_{\abs{\sca{y}{\theta}-t}\leq\lambda},
\end{align*}
where $\omega_{d-1}$ is the volume of the unit ball in $\Rsp^{d-1}$, and where we used that
$\chi$ is supported in $B_1$. Combining this inequality with \eqref{eq:psi-prime} and
$\abs{\partial_\theta f}\leq\Lip(f)$ a.e., we get
$$ \abs{(\psi_\theta)'(t)} \leq \frac{c_d}{\lambda^d}\nr{\partial_\theta^{d-1}\chi}_\infty \Lip(f)\,
   \mathrm{Leb}\big(\{y\in B_A : \abs{\sca{y}{\theta}-t}\leq\lambda\}\big)
   \leq c_d\frac{A^{d-1}}{\lambda^{d-1}}\Lip(f), $$
since this slab has width $2\lambda$ and cross-sections of $(d-1)$-volume at most
$\omega_{d-1}A^{d-1}$. Here and in the sequel, $\mathrm{Leb}$ denotes the Lebesgue measure. Let $\mu_\theta, \nu_\theta$ denote the projections of $\mu,\nu$ on $\R\theta$.
Then, we use Fubini's theorem and Kantorovich-Rubinstein's inequality,
  \begin{align*}
    \abs{\int \psi \dd(\mu-\nu)}
    &= \abs{\int_{\Rsp^d}\int_{\Sph^{d-1}} \psi_\theta(\sca{x}{\theta})\dd\sigma(\theta) \dd(\mu-\nu)(x)}\\ 
    &\leq  \int_{\Sph^{d-1}} \abs{\sca{\psi_\theta}{\mu_\theta - \nu_\theta}} \dd\sigma(\theta)\\
    &\leq \left(\sup_{\theta} \Lip(\psi_\theta)\right) \SlicedWass_1(\mu,\nu) 
    \leq c_d \frac{A^{d-1}}{\lambda^{d-1}} \Lip(f) \SlicedWass_1(\mu,\nu).
  \end{align*}
  Together with \eqref{eq:moll-error}, this implies
$$ \abs{\int f\dd(\mu-\nu)}  \leq 2\Lip(f)\lambda
  +c_d\Lip(f) A^{d-1}\lambda^{-(d-1)} \SlicedWass_1(\mu,\nu).
$$
Choosing $\lambda$ in a way that the two terms in the right-hand side are equal, we get the desired estimate.
\end{proof}

\begin{proof}[Proof of \Cref{t:comparW1sW1}]
  The first inequality in \eqref{eq:sliced-1-comparison} is well-known
  and follows from the fact that $P_\theta(x) := \sca{x}{\theta}$ is
  $1$-Lipschitz. We now prove the second inequality, starting with the
  odd-dimensional case. Let $\phi: \Rsp^d\to\Rsp$ be an optimal
  Kantorovich potential associated to the distance
  $\Wass_1(\mu,\nu)$. In particular, $\phi$ is $1$-Lipschitz and we
  may assume $\phi(0)=0$. Choose a $1$-Lipschitz function
  $\eta:[0,\infty)\to[0,1]$ such that $\eta\equiv 1$ on $[0,1]$ and
    $\eta \equiv 0$ on $[2,\infty)$, and define the truncated
      potential
\[
  \phi_L(x):=\eta\!\left(\frac{|x|}{L}\right)\phi(x).
\]
The function $\phi_L$ coincides with $\phi$ on $B_L$, it is
supported on $B_{2L}$ and $\Lip(\phi_L) \leq 3$.  Moreover, using
$\phi_L(0) = \phi(0) = 0$ we get $\abs{\phi(x)} + \abs{\phi_L(x)} \leq
4 |x|$, implying
\begin{align}
  \abs{\int(\phi-\phi_L)\dd(\mu-\nu)} 
  &\leq \int_{\Rsp^d\setminus B_L} \abs{\phi}+\abs{\phi_L}\dd(\mu-\nu)\notag\\
  &\leq 4 \int_{\Rsp^d\setminus B_L} |x|\dd(\mu+\nu)(x)\notag\\
  &\leq 4 \int_{\Rsp^d\setminus B_L} |x|^q L^{1-q} \dd(\mu+\nu)(x)\notag\\
  &\leq 4 L^{1-q}(M_q(\mu) + M_q(\nu)) \notag\\
  &\leq 8 R^qL^{1-q}.\label{eq:tail}
\end{align}
Applying Lemma \ref{lem:carlierfigallimerigotwang-test} to $\phi_L$, which is supported on $B_{2L}$ (i.e. $A= 2L$) gives us
\begin{equation}\label{eq:compact-part}
  \abs{\int\phi_L\dd(\mu-\nu)}
  \leq c_d L^{\frac{d-1}{d}}\SlicedWass_1(\mu,\nu)^{\frac1d}.
\end{equation}
Combining \eqref{eq:tail}, \eqref{eq:compact-part} and the optimality
of $\phi$ for $\Wass_1(\mu,\nu)$, we get for every $L>0$,
\begin{equation}\label{eq:before-opt}
  \Wass_1(\mu,\nu) = \abs{\int \phi\dd(\mu - \nu)}
  \leq 8 R^qL^{1-q}
  +c_d L^{1-\frac{1}{d}} \SlicedWass_1(\mu,\nu)^{\frac1d}.
\end{equation}
Choosing $L$ in a way that the two terms in the right-hand side are equal leads to \eqref{eq:sliced-1-comparison}.  The
case where $d$ is even can be reduced to the odd-dimensional case by
embedding the problem in $\Rsp^{d+1}$.
\end{proof}

The generalization of this result to other exponents $p>1$ relies on
Lemma~\ref{l:wassboundedmoments}, which generalizes the well known
fact that on a bounded metric space with diameter $D$ and any $r\geq
p\geq 1$, the $\Wass_p$ and $\Wass_r$ distances are bi-Hölder
equivalent with $\Wass_p \leq \Wass_r \leq D^{1-p/r} \Wass_p^{r/p}$.

\begin{lemma}\label{l:wassboundedmoments}
  Let $(X,d)$ be a metric space, let $\mu,\nu\in\mathcal{M}_q(X)$ and let
  $x_0\in X$. Let  $1 \leq p \leq r < q$ and assume that for some $R>0$,
  $$ \int d(x,x_0)^q \dd\mu(x) \leq R^q, \int d(y,x_0)^q \dd\nu(y)\leq R^q.$$
  Then, with $\alpha = \frac{q(r-p)}{r(q-p)}$ and $1-\alpha = \frac{p(q-r)}{r(q-p)}$ one has
  $$ \Wass_p(\mu,\nu) \leq \Wass_r(\mu,\nu) \leq (2R)^\alpha \Wass_p(\mu,\nu)^{1-\alpha}.
  $$
\end{lemma}

\begin{proof}
The inequality on the left holds without any assumption; let us prove
the other inequality. The exponent $\alpha$ in the statement is chosen
such that $\frac{1}{r} = \frac{1-\alpha}{p} + \frac{\alpha}{q}$. The
assumption $1 \leq p \leq r < q$ implies that $0 \leq \alpha < 1$.
Let $\gamma \in \Gamma(\mu,\nu)$ be any transport plan between $\mu$
and $\nu$.  Hölder's inequality gives $\nr{d}_{\LL^r(\gamma)} \leq
\nr{d}_{\LL^p(\gamma)}^{1-\alpha} \nr{d}_{\LL^q(\gamma)}^{\alpha}.$
Moreover, using the triangle and Minkowski's inequalities,
  \begin{align*}
    \nr{d}_{\LL^q(\gamma)}
    &\leq \nr{d(\cdot,x_0)}_{\LL^q(\gamma)} + \nr{d(x_0,\cdot)}_{\LL^q(\gamma)}\\
    &= \nr{d(\cdot,x_0)}_{\LL^q(\mu)} + \nr{d(x_0,\cdot)}_{\LL^q(\nu)} \leq 2R
  \end{align*}
Combining these two facts and taking a minimizing sequence of transport plans 
$\gamma_n$ with respect to the cost $\Wass_p$ we obtain the desired conclusion.
\end{proof}

We recall the notation
$$
\mathcal{M}_{q,R} := \{ \mu\in\Prob(\Rsp^d) \mid M_q(\mu) \leq R^q\}.
$$
\begin{corollary}\label{c:comparWSW} Let $1\leq p \leq r < q$. Then
\begin{itemize}
\item the metrics $\Wass_p$ and $\Wass_r$ are bi-Hölder equivalent on $\mathcal{M}_{q,R}$:
  for $\alpha=\frac{q(r-p)}{r(q-p)}$, one has
  $$ \forall \mu,\nu\in\mathcal{M}_{q,R}, \qquad \Wass_p(\mu,\nu) \leq \Wass_r(\mu,\nu) \leq (2R)^\alpha \Wass_p(\mu,\nu)^{1-\alpha}.$$
\item the metrics $\SlicedWass_p$ and $\SlicedWass_r$ are bi-Hölder equivalent for the same $\alpha$,
    $$ \forall \mu,\nu\in\mathcal{M}_{q,R}, \qquad \SlicedWass_p(\mu,\nu) \leq \SlicedWass_r(\mu,\nu) \leq (2R)^\alpha \SlicedWass_p(\mu,\nu)^{1-\alpha}.$$
\end{itemize}
\end{corollary}

\begin{proof} The first claim follows directly from the previous lemma.
  Let us prove the sliced version.  For $\theta\in\Sph^{d-1}$, denote
  by $\mu_\theta$ and $\nu_\theta$ the projections of $\mu$ and $\nu$
  on $\Rsp\theta$. By definition, for $s\geq1$,
  $$\SlicedWass_s(\mu,\nu) =
  \nr{\Wass_s(\mu_\cdot,\nu_\cdot)}_{\LL^s(\sigma)},$$ where $\sigma$
  is the uniform probability measure on $\Sph^{d-1}$.  Using the
  pointwise inequality $\Wass_p(\mu_\theta,\nu_\theta)\leq
  \Wass_r(\mu_\theta,\nu_\theta)$ and then the monotonicity of $L^s$
  norms on the probability space $(\Sph^{d-1},\sigma)$, we obtain
  $\SlicedWass_p \leq \SlicedWass_r$. For the converse estimate, we
  use the bounds on the moments of $\mu$ and $\nu$. Since the
  orthogonal projection onto $\Rsp\theta$ is $1$-Lipschitz, we have
  $M_q(\mu_\theta) \leq M_q(\mu) \leq R^q$ and similarly
  $M_q(\nu_\theta) \leq R^q$.  Therefore, applying the previous lemma
  to $\mu_\theta$ and $\nu_\theta$, we obtain, for every
  $\theta\in\Sph^{d-1}$,
  \[\Wass_r(\mu_\theta,\nu_\theta) \leq (2R)^\alpha
    \Wass_p(\mu_\theta,\nu_\theta)^{1-\alpha}.  \]
    Integrating this inequality gives us
\begin{align*}
  \SlicedWass_r(\mu,\nu)
  =
  \left(
    \int_{\Sph^{d-1}}
    \Wass_r(\mu_\theta,\nu_\theta)^r
    \dd\sigma(\theta)
  \right)^{1/r}
  \leq
  (2R)^\alpha
  \left(
    \int_{\Sph^{d-1}}
    \Wass_p(\mu_\theta,\nu_\theta)^{r(1-\alpha)}
    \dd\sigma(\theta)
    \right)^{1/r}.
\end{align*}
We now remark that $r(1-\alpha) = \frac{p(q-r)}{q-p} \leq p$.
Therefore, again by monotonicity of $L^s(\sigma)$ norms with respect to $s$,
\begin{align*}
  \left( \int_{\Sph^{d-1}}\Wass_p(\mu_\theta,\nu_\theta)^{r(1-\alpha)} \dd\sigma(\theta)
  \right)^{1/(r(1-\alpha))}
  \leq \left( \int_{\Sph^{d-1}} \Wass_p(\mu_\theta,\nu_\theta)^p \dd\sigma(\theta) \right)^{1/p} =
  \SlicedWass_p(\mu,\nu).
\end{align*}
Raising this inequality to the power $1-\alpha$ concludes the proof.
\end{proof}

Combining Theorem \ref{t:comparW1sW1} and Corollary \ref{c:comparWSW} we obtain Theorems \ref{t:compareWpSWpholder} and \ref{t:sw}.


\subsection{Non-equivalence of $\Wass_p$ and $\SlicedWass_p$ on $\mathcal{M}_{p,R}$}\label{s:nonequivalence}

In this paragraph, we construct a counter-example showing that $\Wass_p$ and $\SlicedWass_p$ are not equivalent on the set $\Mom_{p,R}$ for $p \in\N\cap[1,+\infty)$.

\begin{theorem}\label{th:nonequiv-sw}
  Let $p\geq 1$ be an integer and let $d\geq 2$. Then, there exist
  probability measures $(\mu_R)_{R>0}$ and $(\nu_{R})_{R>0}$ on
  $\Rsp^d$ such that $M_p(\mu_R) = M_p(\nu_R) = 1$ and such that
$$\lim\inf \Wass_p(\mu_R,\nu_R) > 0 \quad \hbox{ and } \quad \lim_{R\rightarrow +\infty}\SlicedWass_p(\mu_R,\nu_R) = 0.$$
\end{theorem}

Given an exponent $p\geq 1$ and $R>0$, we define a probability measure
$\rho_R$ on $\Rsp_+$ by
\begin{align*}
  &\rho_R = (1-m_R)\delta_0 + \mathbf{1}_{[1,R]} \frac{1}{\log(R)} \frac{\mathrm{Leb}}{r^{p+1}}, 
  \hbox{ where } m_R = \frac{1}{p \log(R)}\left(1 - \frac{1}{R^p}\right).
\end{align*}
With this choice, $\rho_R$ is a probability measure and its $p$-th moment
equals one. Given two probability measures $\mu,\nu$ on the unit
sphere, which we will choose later, we can use the map $(r,x)\in
\Rsp_+\times \Sph^{d-1}\mapsto rx\in\Rsp^d$ to construct a probability
measure over $\Rsp^d$, 
\begin{equation} \label{eq:muR-nuR}
  \mu_R = (r,x)\mapsto rx\vert_{\#}[\rho_R \otimes \mu], \qquad
  \nu_R = (r,x)\mapsto rx\vert_{\#}[\rho_R \otimes \nu]
\end{equation}
Here $\rho_R$ is the radial part and $\mu$ (or $\nu$) the angular part. Looking for measures 
decomposed according to a radial and an angular part is natural in view of the definition of the sliced-Wasserstein distance. We note that with this definition, $M_p(\mu_R) = M_p(\nu_R) =
M_p(\rho_R) = 1$. Theorem ~\ref{th:nonequiv-sw} will then follow directly from Lemmas \ref{l:distinctmunu}, \ref{l:siCp+egaux} and \ref{l:Cp+egaux}.

\begin{lemma}\label{l:distinctmunu}
Let $\mu,\nu$ be distinct probability measures on the sphere $\Sph^{d-1}$. Then 
$$\liminf_{R\to+\infty}\Wass_p(\mu_R,\nu_R) > 0.$$
\end{lemma}

\begin{proof} 
Let $g \in\mathcal{C}^1(\Sph^{d-1})$ be $1$-Lipschitz and assume that
$g(x_0) = 0$ for some point $x_0\in\Sph^{d-1}$. We consider its
$p$-homogeneous extension $f: \Rsp^d\to\Rsp$; precisely, $f(0) = 0$
and $f(x) = |x|^p g(x/|x|)$ for $x\neq0$. Then $f$ has class
$\mathcal{C}^1$ on $\Rsp^d\setminus\{0\}$ and has gradient
$$ \nabla f(x) = |x|^{p-2} x\, g(x/|x|) + |x|^{p-1}\nabla
g(x/|x|), $$ where $\nabla g$ is the tangential gradient.  Hence,
$|\nabla f(x)|\leq (2p+1)|x|^{p-1}$. If $x,y$ are two points
such that the segment $[x,y]$ does not pass through the origin, we get
by integration,
$$ \abs{f(x)-f(y)} \leq (2p+1) \max(|x|,|y|)^{p-1}|x-y|, $$
and by continuity of $f$ the same inequality holds for all
$x,y\in\Rsp^d$.  Let $\gamma$ be an optimal transport plan between the
two measures $\mu_R$ and $\nu_R$ for the cost $c_p$. H\"older's
inequality  gives us
\begin{align*}
\abs{\int f\dd(\mu_R-\nu_R)}
&\leq (2p+1) \int \max(|x|,|y|)^{p-1}|x-y|\dd\gamma(x,y) \\
&\leq (2p+1) \Big(\int (|x|^p+|y|^p)\dd\gamma\Big)^{\frac{p-1}{p}} \Wass_p(\mu_R,\nu_R)\\
&= (2p+1) \, (M_p(\mu_R) + M_p(\nu_R))^{1-\frac1p}\, \Wass_p(\mu_R,\nu_R),\\
&\leq 2 (2p+1)  \Wass_p(\mu_R,\nu_R)
\end{align*}
On the other hand, by homogeneity and using Fubini,
\[
\int f\dd(\mu_R -\nu_R)
= \int_{\Rsp_+} r^p\dd\rho_R(r)\,\int g\dd(\mu-\nu)
= \int g\dd(\mu-\nu).
\]
Combining the two equations
above with the Kantorovich-Rubinstein duality formula we get
\[\Wass_1(\mu,\nu) = \sup_{g \in \mathcal{C}^1(\Sph^{d-1}),~\Lip(g)\leq 1} \int g\dd(\mu-\nu) \leq 2(2p+1)  \Wass_p(\mu_R,\nu_R).\]
This proves that $\mu_{R}$ and $\nu_R$ remain at positive $\Wass_p$
distance.
\end{proof}


The following lemma will be useful to control the distance between the projections of $\mu_R$ and $\nu_R$ in a direction $\theta$. Given a probabily measure $\lambda$ on $\Rsp$ and $p\geq 1$, we will denote
$$M_{p,\pm}(\lambda) = \int a_\pm^p \dd\lambda(a) $$
where $a_+ = \max(a,0)$ and $a_- = -\min(a,0)$ denote the positive and negative parts of $a$. Finally, given $a\in\Rsp$, we let  $D_a: \Rsp\to\Rsp: x \mapsto ax$ be the dilation by $a$.
\begin{lemma} Given a measure $\lambda\in\Prob([-1,1])$, set 
$\rho_R[\lambda]:=
\int_{[0,1]} (D_a)_\#\rho_R\,d\lambda(a)$. Then, 
$$ \max \bigg\{ \Wass_p^p(\rho_R[\lambda], \rho_R[\lambda']) \mid \lambda,\lambda' \in\Prob([-1,1]), M_{p,\pm}(\lambda) = M_{p,\pm}(\lambda')\bigg\}\leq \frac{c_p}{\log(R)} $$
\end{lemma}
\begin{proof}
Let $\lambda,\lambda'$ be as in the statement. Let
$\phi:\Rsp_+\to\Rsp$ be measurable, with $\phi(0) = 0$ and $\abs{\phi(s)} \leq s^p$.
For $a\neq 0$, the change of variable \(a\in(0,1]\), the change of variables \(s=\abs{a}r\) gives
$$\int\phi \dd(D_{a\#}\rho_R) = \int \phi(ar) \dd\rho_R = \frac{\abs{a}}{\log(R)} \int_{\abs{a}}^{\abs{a}R} \phi(\sgn(a)s)\frac{\dd s}{s^{p+1}}$$
Integrating this inequality against $\lambda$, we obtain
\[
\int \phi\, \dd \rho_R[\lambda] 
= \frac{1}{\log R} \int_{[-1,1]} 
\abs{a}^p \left( \int_{\abs{a}}^{\abs{a}R}\phi(\sgn(a) s)\frac{ds}{s^{p+1}}\right)
\dd\lambda(a),
\]
Using $M_{p,\pm}(\lambda) = M_{p,\pm}(\lambda')$, this implies
\begin{align*}
    \int \phi\, \dd (\rho_R[\lambda] - \rho_R[\lambda']) 
    &= \frac{1}{\log R} \int_{[-1,1]} \abs{a}^p \left( \int_{\abs{a}}^{\abs{a}R}\phi(\sgn(a)s)\frac{ds}{s^{p+1}}\right) 
\dd(\lambda-\lambda')(a)\\
    &= \frac{1}{\log R} \int_{[-1,1]} \abs{a}^p \left( \int_{\abs{a}}^{\abs{a}R}\phi(\sgn(a)s)\frac{ds}{s^{p+1}} - \int_1^{R}\phi(\sgn(a)s)\frac{ds}{s^{p+1}}\right) 
\dd(\lambda-\lambda')(a).
\end{align*}
For every $a\neq0$, using $\abs{\phi(s)}\leq \abs{s}^p$,
$$
\left|\int_{\abs{a}}^{\abs{a}R}\phi(\sgn(a)s)\frac{ds}{s^{p+1}} - \int_1^R\phi(\sgn(a)s)\frac{ds}{s^{p+1}}\right|
\le \int_{\abs{a}}^1 \frac{ds}{s} + \int_{\abs{a}R}^R \frac{ds}{s} = 2|\log a|.
$$
The two previous inequalities imply, letting $c_p = \sup_{0<a\leq 1} a^p \abs{\log a} < +\infty$,
$$ \abs{\int\phi\,d\bigl(\rho_R[\lambda]-\rho_R[\lambda']\bigr)} \leq \frac{4c_p}{\log R}.$$
Let $\chi$ be the density of the measure $\rho_R[\lambda] - \rho_R[\lambda']$ with respect to the total variation measure $\abs{\rho_R[\lambda] - \rho_R[\lambda']}$ and let $\phi(s)=s^p\chi(s)$. We obtain
$$\int_{\Rsp_+} s^p \dd\abs{\rho_R[\lambda]-\rho_R[\lambda']}(s) \leq \frac{4c_p}{\log R}.$$
By Theorem 6.15 in \cite{villani2009optimal}, this inequality yields the desired  bound on the $\Wass_p$ distance:
\begin{equation*} 
\Wass_p^p(\rho_R[\lambda], \rho_R[\lambda']) \leq 2^{p-1} \int \abs{s}^p \dd\abs{\rho_R[\lambda]-\rho_R[\lambda']} \leq \frac{2^{p+1} c_p}{\log(R)}. \qedhere
\end{equation*}
\end{proof}

Given a measure $\mu\in\Prob(\Sph^{d-1})$ and $p\geq 1$, we define a
Lipschitz function $C_p^+[\mu]$ on $\Sph^{d-1}$ by
$$ C_p^+[\mu](\theta) := \int \sca{\sigma}{\theta}_+^p
\dd\mu(\sigma),$$ where $t_+ = \max(t,0)$ denotes as usual the
positive part of $t\in\Rsp$. This is closely related to the
\emph{spherical $p$-cosine transform} in convex geometry, which is
defined with the same formula with an absolute value instead of the
positive part.

\begin{lemma} \label{l:siCp+egaux} Let $p\geq 1$ and assume that $\mu,\nu \in\Prob(\Sph^{d-1})$ satisfy
  $C_p^+[\mu] = C_p^+[\nu]$. Then,
\[
  \SlicedWass_p^p(\mu_R,\nu_R) \leq 
  \sup_{\theta\in\Sph^{d-1}}
  W_p^p\bigl((\theta\cdot)_\#\mu_R,(\theta\cdot)_\#\nu_R\bigr)
  \underset{R\rightarrow +\infty}{\longrightarrow}0.
\]
\end{lemma}

\begin{proof}
Let $\pi \in \{\mu,\nu\}$ be one of the two measures on $\Sph^{d-1}$ and let
$\pi_R\in\{\mu_R,\nu_R\}$ be the corresponding measure on
$\Rsp^d$. By definition of the pushforwards, sampling from $\pi_{R,\theta}$ amounts to sampling a radius $r\sim \rho_R$, sampling a direction $\sigma \sim \pi$, and then considering $\sca{\sigma}{\theta} r$. This shows that $\pi_{R,\theta}$ is a mixture of dilations of $\rho_R$. More explicitely,
$$ \pi_{R,\theta} = \rho_R[\lambda_{\pi,\theta}], \hbox{ where } \lambda_{\pi,\theta} = P_{\theta\#} \pi.$$
One may check that $M_{p,+}(\lambda_{\pi,\theta}) = C_p^+[\pi](\theta)$ and $M_{p,-}(\lambda_{\pi,\theta}) = C_p^+[\pi](-\theta)$. By assumption, $C_p^+[\mu] = C_p^+[\nu]$, and we may therefore apply the previous lemma to get
$$ \Wass_p^p(\mu_{R,\theta}, \nu_{R,\theta}) \leq \frac{c_p}{\log(R)}.$$
This directly implies that 
$$ 0 \leq  \SlicedWass_p(\mu_R,\nu_R) \leq \max_{\theta\in\Sph^{d-1}} \Wass_p(\mu_{R,\theta},\nu_{R,\theta}) \leq \frac{c_p}{\log(R)}, $$
so that $\lim_{R\to+\infty}\SlicedWass_p(\mu_R,\nu_R) = 0$.
\end{proof}

\begin{lemma}\label{l:Cp+egaux}
Let $d\geq 2$, let $p\geq 1$ be an integer, and let $\mu$ be the uniform probability
measure on $\Sph^{d-1}$. Then there exists $\nu\in\Prob(\Sph^{d-1})\setminus\{\mu\}$ such that
$C_p^+[\mu] = C_p^+[\nu]$.
\end{lemma}

\begin{proof}
Let $\mathcal{P}_p\subseteq \LL^2(\mu)$ be the space of restrictions
of real polynomials of degree $\leq p$ to the unit sphere, and let
$\Pi$ be the orthogonal projection onto $\mathcal{P}_p$. Define $q(x)
= x_1^{p+2}$ and let $h = q - \Pi q$. Since $q$ does not belong to $\mathcal{P}_p$, the function $h$ is nonzero.
The operator
$$ S: f \in \LL^2(\mu)\mapsto f\circ (-\Id) \in \LL^2(\mu)$$ is an
isometry of $\LL^2(\mu)$ which leaves $\mathcal{P}_p$ invariant and commutes with $\Pi$. Since
$Sq = (-1)^{p+2} q$, we deduce that $Sh = (-1)^p h$.  We observe that $t_+^p =
\frac12(t^p + s(t))$, with $s(t) = \abs{t}^{p-1}t$ if $p$ is even
and $s(t) = \abs{t}^p$ if $p$ is odd. Also, for any
$\theta\in\Sph^{d-1}$, $S s_\theta = (-1)^{p+1} s_\theta$, where
$s_\theta(\sigma) = s(\sca{\theta}{\sigma})$.  Then, by orthogonality
between $h$ and $\mathcal{P}_p$ and, one has for all $\theta
\in\Sph^{d-1}$,
\begin{align*}
  C_p^+[h\mu](\theta) &= \int \sca{\theta}{\sigma}_+^p h(\sigma) \dd\mu(\sigma)
  = \frac{1}{2}\int s_\theta h \dd \mu \\
  &= \frac{1}{2}\int S s_\theta S h \dd\mu 
  = \frac{1}{2}\int (-1)^{p+1}s_\theta  (-1)^p h \dd\sigma = - \frac{1}{2}\int s_\theta h \dd \mu 
  = 0,
\end{align*}
where the third equality follows from the application of the operator
$S$ to both functions.  This implies that $C_p^+[h\mu] = 0$. For
$\varepsilon>0$ small enough, namely $\varepsilon \leq
\nr{h}_\infty^{-1}$, the measure $\nu = (1+\varepsilon h)\mu$ is
non-negative and distinct from $\mu$. Since the constant function $1$
belongs to $\mathcal{P}_p$, the orthogonality between $h$ and
$\mathcal{P}_p$ implies that $\nu$ is a probability measure. Finally,
$C_p^+[\nu] = C_p^+[\mu + \varepsilon h\mu] = C_p^+[\mu]$.
\end{proof}

\begin{proof}[Proof of \Cref{th:nonequiv-sw}] Since $p\geq 1$ and $p\in\mathbb{N}$, by \Cref{l:Cp+egaux},  
there exists $\nu\in\Prob(\Sph^{d-1})\setminus\{\mu\}$ such that $C_p^+[\mu] = C_p^+[\nu]$. Since $\mu\neq \nu$,  \Cref{l:distinctmunu} guarantees that $\liminf_{R\to+\infty} \Wass_p(\mu_R,\nu_R)>0$ while \Cref{l:siCp+egaux} implies that $\lim_{R\to+\infty}\SlicedWass_p(\mu_R,\nu_R) = 0$, as desired.
\end{proof}

\section{Proofs for kernel mean embeddings}\label{s:mmd}

\begin{proof}[Proof of Theorem \ref{t:MMDW1}]
We first observe that proving the case $p=1$ is sufficient: indeed, once this case is proven, the case $p>1$ follows by applying Lemma \ref{l:wassboundedmoments} with $p=1$. In the sequel, we assume therefore that $p=1$.

We notice that \eqref{e:W1smallerkalpha} is directly given by \cite[Proposition 17]{modeste}. Regarding \eqref{e:kalphasmallerW1}, we rely on \cite[Lemma 20]{modeste} which says that for all $\mu,\nu$ with finite $\beta$-th moment, there holds
$$
\|K_\mu-K_\nu\|_{\mathcal{H}_{\kappa_\beta}}^2\leq (M_\beta(\mu)+M_\beta(\nu))W_\beta(\mu,\nu).
$$
Let $\mu,\nu\in\mathcal{M}_{1,R}$. We have
$$
M_\beta(\mu)=\int_{\R^d}|y|^\beta d\mu(y)\leq 1+\int_{\R^d}|y|d\mu(y)\leq R^r+1
$$
and similarly $M_\beta(\nu)\leq R^q+1$. Now, $W_\beta\leq W_1^\beta$ by Hölder's inequality, hence finally 
$$
\|K_\mu-K_\nu\|_{\mathcal{H}_{\kappa_\beta}}\leq (2R^q+2)^{\frac12}W_1(\mu,\nu)^{\frac{\beta}{2}}
$$
which concludes the proof.
\end{proof}


\begin{proof}[Proof of Proposition \ref{p:nouniformMMD}]
Let $e\in\R^d$, $|e|=1$. We consider the case $R=1$, the general case follows easily. Set $\mu_n=\frac{1}{2n^p} (\delta_{ne}+\delta_{-ne})+\left(1-\frac{1}{n^p}\right)\delta_0$ and $\mu=\delta_0$. Note that these measures have first moment equal to $1$. We show that $d_\beta(\mu_n,\mu)\rightarrow0$. We have $\kappa_\beta(0,\cdot)=0$ and $\kappa_\beta(ne,ne)=\kappa_\beta(-ne,-ne)=2n^{2\beta}$ and $\kappa_\beta(ne,-ne)=(2-2^{2\beta})n^{2\beta}$ hence
$$
\|K_{\mu_n}-K_\mu\|_{\mathcal{H}_{\kappa_\beta}}^{2p}=\frac{1}{2n^p}[\kappa_\beta(ne,ne)+\kappa_\beta(-ne,-ne)+2\kappa_\beta(ne,-ne)]=(4-2^{2\beta})n^{2\beta-2p}.
$$
Since $W_p(\mu_n,\mu)=1$ and $\beta<1\leq p$ we get the result. 
\end{proof}

\appendix

\section{Stability of $q$-cost optimal transport maps} \label{s:stabqcost}

Let $\rho$ be an absolutely continuous probability measure, whose density is bounded from above and below on its convex and compact support. Given  $p>1$ and $\mu\in\Prob_p(\Rsp^d)$, we denote $\phi_{\rho\to\mu,p}$ the $c_p$-concave Kantorovich potential on the side of $\rho$ that has zero $\rho$-mean. We denote as before $T_{\rho\to\mu,p}$  the induced transport map and  $\psi_{\rho\to\mu}$ the $c_p$-transform of $\phi_{\rho\to\mu,p}$ introduced in \eqref{e:cptransform}.  We show the following result, which implies \eqref{e:mischlertrevisan}.


\begin{theorem} \label{t:mt}
Let $d\geq 1$ and let $\rho$ be a probability measure on $\R^d$ with bounded and convex support $\Omega$, absolutely continuous with respect to the Lebesgue measure, with density bounded from above and below by strictly positive constants. 
Let $p>1$, let $\mu,\nu\in\Prob(B(0,R))$. Then, 
$$ \|T_{\rho\rightarrow\mu,p}-T_{\rho\rightarrow\nu,p}\|_{L^2(\rho)} \leq 
C R^{1-(p-1)\delta} \left(\int (\psi_{\rho\to\mu,p} - \psi_{\rho\to\nu,p}) \dd(\nu-\mu)\right)^\delta. $$
When $p\geq 2$, one may take $\delta=\frac{1}{6(p-1)}$, and $C = C(p,\rho)$. When $p\in(1,2)$, for any $\theta\in (0,(p-1)/(p+1))$ one may take $\delta = \theta(1- 1/p)$ and $C = C(p,\rho,\theta)$.
\end{theorem}

\begin{corollary}\label{c:localizedMT}
Let $\rho$ satisfy the assumptions of the previous theorem, and let $p>1$. Let $\mu,\nu\in\Prob_p(\Rsp^d)$ and let $\rho'\leq\rho$. Define
\begin{equation}\label{e:mu'nu'}
\mu'=T_{\rho\to\mu,p\#}\rho',
\qquad
\nu'=T_{\rho\to\nu,p\#}\rho'.
\end{equation}
Assume that $\mu'$ and $\nu'$ are supported in $B(0,R)$ for some $R\geq1$. Then 
\[
\|T_{\rho\to\mu,p}-T_{\rho\to\nu,p}\|_{L^2(\rho')}
\leq
C R^{1-(p-1)\delta}
\left(\int (\psi_{\rho\to\mu,p} - \psi_{\rho\to\nu,p}) \dd(\nu-\mu)\right)^\delta,
\]
where $\delta$ and $C$ are as in the previous theorem.
\end{corollary}

\begin{proof}[Proof of \Cref{t:mt}]
The above statement is a consequence of \cite[Theorem 1.2]{mischlertrevisan}, except that the polynomial dependence in $R$ on \cite[Theorem 1.2]{mischlertrevisan} is not written explicitly there. We explain how to track  this dependence  in the proof of \cite[Theorem 1.2]{mischlertrevisan}. For simplicity, we write $T_\mu$ instead of $T_{\rho\rightarrow \mu,p}$.  Notice that it does not seem obvious to apply directly an homogeneity argument, unlike in the quadratic case: optimality of the map $T$ from $\rho$ to $\mu$ for $c_p$ does not imply that $L T$ is optimal from $\rho$ to $(LT)_{\#}\rho$ except when $p=2$.

Recall that 
\begin{equation}\label{e:otpotTmu}
T_\mu=\Id - \nabla \phi_\mu |\nabla\phi_\mu|^{\frac{1}{p-1}-1}   
\end{equation}
for some function $\phi_\mu$ \cite[Theorem 6.2.4]{ambrosio2005gradient}. Up to adding a constant to $\phi_\mu$ we may assume that $\phi_\mu$ has zero $\rho$-mean. 

The proof of Theorem \cite[Theorem 1.2]{mischlertrevisan} relies on \cite[Theorem 1.1]{mischlertrevisan}. Therefore, we need to trace the dependence of the constant $C$ appearing in the statement of \cite[Theorem 1.1]{mischlertrevisan}:
we prove that under the assumptions of \cite[Theorem 1.1]{mischlertrevisan}, 
\begin{equation}\label{e:stabpotavecR}
\|\phi_\mu-\phi_\nu\|_{L^2(\rho)}\leq A'R^{n'}W_1(\mu,\nu)^{\delta'}
\end{equation}
for some $A',n',\delta'>0$. The  proofs of both \cite[Theorem 1.1]{mischlertrevisan} and \cite[Theorem 1.2]{mischlertrevisan} separate the cases $p\geq 2$ and $1<p<2$. We also follow this dichotomy. 

An important observation, valid for any $p>1$, is that the Kantorovich potentials $\phi_\mu$ and $\psi_\mu$ have  Lipschitz constant $\leq C_\rho R^{p-1}$. This follows from the fact that 
\begin{equation}\label{e:cconv}
\phi_\mu(x)=\inf_{y\in B(0,R)} |y-x|^p-\psi_\mu(x)
\end{equation}
(and similarly for $\psi_\mu$ expressed in terms of $\phi_\mu$), see the first lines of the proof of Lemma \ref{l:boundW1-qmoment}.

\medskip

\emph{Case $p\geq 2$.} It is proved in \cite[Theorem 3.6]{mischlertrevisan} that if $\rho$ is bounded from above and below on a compact convex set with respect to  the Lebesgue measure (or more generally a log-concave measure), then 
$$
\|\phi_\mu-\phi_\nu\|_{L^2(\rho)}^2\leq C_{p,\rho}R^{p-1} \int (\psi_\mu-\psi_\nu) \dd(\nu-\mu)
$$
where $C_{p,\rho}$ only depends on $p$ and $\rho$. In the left-hand side, we used the fact that  $\phi_\mu$ and $\phi_\nu$ have zero $\rho$-mean to replace the variance by an $L^2(\rho)$ norm.

Let us explain how to turn this into an upper bound on $\|T_\mu-T_\nu\|_{L^2(\rho)}$, for $p\geq 2$ (similar to \cite[Theorem 4.3]{mischlertrevisan}, but with explicit constants). When $p=2$, we use \cite[Proposition 4.1]{mischlertrevisan} (also reported as \cite[Proposition 4.2]{mischlertrevisan}) and the fact that $\phi_\mu, \phi_\nu$ are Lipschitz with constant $\leq C_\rho R$ to obtain 
\begin{equation}\label{e:p-1partout}
\|T_\mu-T_\nu\|_{L^2(\rho)}\leq C_{p,\rho} R\left( \int (\psi_\mu-\psi_\nu) \dd(\nu-\mu)\right)^{\frac16}.
\end{equation}
Let us now consider the case $p>2$. The heart of the proof of \cite[Theorem 4.3]{mischlertrevisan} is that there exists
$\gamma>0$ such that, for any function $\phi$ that is $c$-concave with respect to the $p$-cost, the function
$x\mapsto \phi(x)-\gamma|x|^2$ is concave and Lipschitz. One can actually take $\gamma= C_{p,\rho}R^{p-2}$ due to the definition \eqref{e:cconv} of $c$-concave functions (for any $\phi_\mu,\phi_\nu$). Using again that $\phi_\mu$ and $\phi_\nu$ are Lipschitz with constant $C_\rho R^{p-1}$, we get that the constant $C$ appearing in \cite[Eq. (4.8)]{mischlertrevisan} can be taken to be $R^{\frac43(p-1)}$, meaning that we obtain, when $p>2$,
\begin{equation}\label{e:p-1partoutp}
\|T_\mu-T_\nu\|_{L^2(\rho)}\leq C_{p,\rho} R^{\frac56}\left( \int (\psi_\mu-\psi_\nu) \dd(\nu-\mu)\right)^{\frac{1}{6(p-1)}}.
\end{equation}

\medskip

\emph{Case $1<p<2$.}
To establish \eqref{e:stabpotavecR} in case $1<p<2$, we follow \cite[Section 3.3]{mischlertrevisan}. We start from \cite[Eq. (3.68)]{mischlertrevisan}:
\begin{equation}\label{e:phi0phi1hbeta}
\min\left(\|\phi_\mu-\phi_\nu\|_{L^2(\rho)}^2,\|\phi_\mu-\phi_\nu\|_{L^2(\rho)}^{\frac{p}{p-1}}\right)\leq \frac{1}{\sup_{\beta\in\R^+} h(\beta)}\int (\psi_\nu-\psi_\mu)\dd(\mu-\nu).
\end{equation}
The function $h$ which appears in this equation is
\begin{equation}\label{e:hbeta}
h(\beta)=\beta e^{-\alpha\beta}-C e^{\alpha'\beta}\beta^p.
\end{equation}
The parameters involved in this definition of $h$ are the following. First, the constant $C$ in \eqref{e:hbeta} is defined as $C=\gamma_p C_\rho$ where
$\gamma_p$ is defined in \cite[Eq. (3.36)]{mischlertrevisan} and depends only on $p$, and $C_\rho$ depends only on $\rho$ and comes from \cite[Lemma 2.2]{mischlertrevisan}. Second, $\alpha=2M_{\phi_\mu,\phi_\nu}$, and $\alpha'=\frac{5p-2}{2}M_{\phi_\mu,\phi_\nu}$ where
$M_{\phi_\mu,\phi_\nu}= \|\phi_\mu\|_{L^\infty(\rho)}+ \|\phi_\nu\|_{L^\infty(\rho)}$.  
Since $\phi_\mu$ is Lipschitz with constant $C_\rho'R^{p-1}$ and has zero $\rho$-mean, it follows that $\|\phi_\mu\|_{L^\infty(\rho)}\leq C_\rho'R^{p-1}$. By symmetry between $\mu$ and $\nu$, we have \begin{equation}\label{e:Mphimuphinu}
M_{\phi_\mu,\phi_\nu}=C_\rho' R^{p-1}
\end{equation}.

We will prove that 
\begin{equation}\label{e:hbetagrand}
\sup_{\beta\in\R^+} h(\beta)\geq c' R^{1-p}
\end{equation}
for some $c'$ depending only on $p,\rho$. For this, pick $\beta=c(1+\alpha)^{-1}$ where $c>0$ is small enough (depending only on $p$ and $\rho$) so that $e^{-c}\geq \frac12$ and $e^{c\frac{5p-2}{4}}\leq 2$ and $2Cc^p\leq \frac{c}{4}$. Then
$$
\begin{aligned}
    \sup_{\beta\in\R^+} h(\beta)\geq h(c(1+\alpha)^{-1}) &= c(1+\alpha)^{-1}e^{-c}-Ce^{c\frac{\alpha'}{\alpha}}c^p(1+\alpha)^{-p}\\
    &\geq \frac{c}{2} (1+\alpha)^{-1}-2Cc^p(1+\alpha)^{-p}\\
    &\geq \frac{c}{4}(1+\alpha)^{-1} \geq c'R^{1-p}.
\end{aligned}
$$
In the last inequality we used \eqref{e:Mphimuphinu} and $R\geq 1$.
Together with \eqref{e:Mphimuphinu}, this establishes \eqref{e:hbetagrand}. 

Plugging into \eqref{e:phi0phi1hbeta} and using that $\psi_\mu$ and $\psi_\nu$ is Lipschitz with a constant $\leq C_\rho'R^{p-1}$, we obtain
\begin{equation}\label{e:minpotW1}
\min\left(\|\phi_\mu-\phi_\nu\|_{L^2(\rho)}^2,\|\phi_\mu-\phi_\nu\|_{L^2(\rho)}^{\frac{p}{p-1}}\right)\leq C_\rho' R^{p-1}\int (\psi_\nu-\psi_\mu)\dd(\mu-\nu).
\end{equation}
Following \cite[Eq. (3.69)]{mischlertrevisan} we observe that $\|\phi_\mu-\phi_\nu\|_{L^2(\rho)}^{\frac{p}{p-1}}\leq C_\rho'R^{2-p}\|\phi_\mu-\phi_\nu\|_{L^2(\rho)}^2$. Plugging into \eqref{e:minpotW1}, we obtain 
\begin{equation}\label{e:1<p<2psinupsimu}
\|\phi_\mu-\phi_\nu\|_{L^2(\rho)}
\leq C_\rho' R^{\frac{(p-1)^2}{p}} \left(\int (\psi_\nu-\psi_\mu)\dd(\mu-\nu)\right)^{1-\frac1p}
\end{equation}
This yields \eqref{e:stabpotavecR} in this case.

We turn to \cite[Theorem 1.2]{mischlertrevisan} which, for $1<p<2$, is \cite[Theorem 4.4]{mischlertrevisan}. Recall that in this statement the parameter $\theta\in (0,(p-1)/(p+1))$ is fixed but arbitrary. We notice that the dependence of the constant $C$ in \cite[Proposition 4.5]{mischlertrevisan} is not correctly written, since as we will see it depends not only on $\theta$, $\X$, $p$, but also on the Lipschitz constants of $u$ and $v$. Let us show that we can take 
\begin{equation}\label{e:47}
C=C_{\theta,\rho,p}(1+{\rm Lip}(u)+{\rm Lip}(v))^{1-\theta}
\end{equation}
in \cite[Eq. (4.10)]{mischlertrevisan}.
First, in \cite[Lemma 4.6]{mischlertrevisan}, all constants only depend on $p$. Then one may take 
\eqref{e:47} in \cite[Lemma 4.7]{mischlertrevisan} (more precisely in \cite[Eq. (4.18)]{mischlertrevisan}). Indeed from the proof of \cite[Eq. (4.22)]{mischlertrevisan}, one has $[u']_{\alpha,1}\leq C(1+{\rm Lip}(u))$ where $C$ does not depend on $u$, and $\|u\|_{W^{1,\infty}}\leq {\rm Lip}(u)$. Plugging into \cite[Lemma 2.8]{mischlertrevisan} we obtain that the constant in \cite[Lemma 4.7]{mischlertrevisan} is \eqref{e:47}. Then using Hölder's inequality as in the conclusion of the proof of \cite[Proposition 4.5]{mischlertrevisan}, we conclude that one may take \eqref{e:47} in \cite[Eq. (4.10)]{mischlertrevisan}.

For the conclusion of the proof of \cite[Theorem 4.4]{mischlertrevisan} we proceed as follows: we replace \cite[Eq. (4.36)]{mischlertrevisan} by 
$$
\begin{aligned}
    \|T_\nu-T_\mu\|_{L^2(\X)}&\leq C_{p,d}(\Lip(\phi_\mu)^{\frac{2-p}{p-1}}+\Lip(\phi_\mu)^{\frac{2-p}{p-1}}) \|\nabla\phi_\nu-\nabla\phi_\mu\|_{L^2(\X)}\\
&\leq C_{p,q}R^{2-p}\|\nabla\phi_\nu-\nabla\phi_\mu\|_{L^2(\X)}
\end{aligned}
$$
where we use again that $\phi_\mu$ and $\phi_\nu$ have Lipschitz constant $\leq C_\rho R^{p-1}$.
Then, \cite[Eq. (4.37)]{mischlertrevisan}
becomes 
$$
\|T_\nu-T_\mu\|_{L^2(\X)}\leq  C_{p,d}R^{2-p}\|\nabla\phi_\nu-\nabla\phi_\mu\|_{L^2(\X)}\leq C_{p,\rho,\theta}R^{2-p}R^{(p-1)(1-\theta)}\|\phi_\mu-\phi_\nu\|_{L^2(\rho)}^\theta
$$
where the last inequality comes from \cite[Proposition 4.10]{mischlertrevisan} with the explicit constant \eqref{e:47}. Finally, using \eqref{e:1<p<2psinupsimu} we get
\begin{align*}
\|T_\nu-T_\mu\|_{L^2(\X)}&\leq C_{p,\rho,\theta}R^{2-p}R^{(p-1)(1-\theta)} R^{\frac{(p-1)^2}{p}\theta} \left(\int (\psi_\nu-\psi_\mu)\dd(\mu-\nu)\right)^{\theta(1-\frac1p)}\\
&=C_{p,\rho,\theta} R^{1-(p-1)\delta} \left(\int (\psi_\nu-\psi_\mu)\dd(\mu-\nu)\right)^{\delta}
\end{align*}
which concludes the proof.
\end{proof}

\begin{proof}[Proof of Corollary \ref{c:localizedMT}]
To lighten notation, we drop the dependence in the source measure $\rho$ and the exponent $p$ in the notation for transport maps and potentials. Let $\rho'\leq \rho$ and let $\mu'=T_{\mu\#}\rho'$, $\nu'=T_{\nu\#}\rho'$. The restrictions of $T_\mu$ and $T_\nu$ to $\spt(\rho')$ are optimal transport maps from $\rho'$ to $\mu'$ and $\nu'$, respectively, and the Kantorovich pairs $(\phi_\mu,\psi_\mu)$ and $(\phi_\nu,\psi_\nu)$ remain optimal for these restricted transport problems.

We note that since
$\phi_{\mu}(x) = \inf_{y} \abs{x-y}^p - \psi_\mu(y)$
and since the minimum is attained at $y=T_\mu(x)$, we have, taking $y = T_\nu(x)$,
$$ \psi_\mu(T_\nu(x))  - \psi_\mu(T_\mu(x)) \leq \abs{x-T_{\nu}(x)}^p -  \abs{x - T_{\mu}(x)}^p.$$
Summing this inequality with the inequality where $\mu$ and $\nu$ are switched, we obtain
$$  \psi_{\mu}(T_\nu(x)) -  \psi_\mu(T_\mu(x))+ \psi_\nu(T_\mu(x)) - \psi_\nu(T_\nu(x)) \leq 0.$$
Integrating this inequality against any non-negative measure $\tilde{\rho}$, we have 
$$ \int (\psi_\mu - \psi_\nu) \dd (T_{\mu\#}\tilde{\rho} - T_{\nu\#}\tilde\rho) \geq 0.$$
Setting $\tilde{\rho} = \rho - \rho'$, this implies in particular the desired monotonicity inequality:
$$ \int (\psi_\mu - \psi_\mu) \dd(\mu'-\nu') \leq \int (\psi_\mu - \psi_\nu) \dd(\mu-\nu).$$
\end{proof}

\bibliographystyle{plain}
\bibliography{references}

\end{document}